\documentclass[reqno,12pt]{amsart}

\usepackage{a4wide,graphicx,amssymb,url}
\usepackage{color,epsf,cancel}
\usepackage{pstricks}
\usepackage{pstricks-add}
\usepackage{pgf,tikz,pgfplots}
\pgfplotsset{compat=1.15}
\usetikzlibrary{positioning}

\usepackage{lmodern}
\usepackage{amsthm}
\usepackage{enumerate}
\usepackage{nicefrac}

\usepackage{framed}

\usepackage{makecell}

\usepackage{bm}

\newtheorem{theorem}{Theorem}[section]
\newtheorem{lemma}[theorem]{Lemma}
\newtheorem{proposition}[theorem]{Proposition}

\newtheorem{corollary}[theorem]{Corollary}
\theoremstyle{definition}
\newtheorem{definition}[theorem]{Definition}

\theoremstyle{remark}
\newtheorem{remark}[theorem]{Remark}

\numberwithin{equation}{section}

\newcommand {\ZZ}  {{\mathbb Z}}
\newcommand {\CC}  {{\mathbb C}}
\newcommand {\RR}  {{\mathbb R}}
\newcommand {\QQ}  {{\mathbb Q}}
\newcommand {\KK}  {{\mathbb K}}

\newcommand{\spk}[1]{\left\langle #1\right\rangle}

\newcommand{\abs}[1]{{\left\lvert#1\right\rvert}}
\newcommand{\nm}[1]{\left\| #1 \right\|}

\renewcommand {\Im}  {\text{Im}}

\newcommand {\fL}  {f_L}
\newcommand {\fR}  {f_R}

\newcommand {\A}  {{\mathcal{A}}}

\newcommand {\R}  {{\mathcal{R}}}

\newcommand {\bfx}  {{\bm{x}}}
\newcommand {\bfy}  {{\bm{y}}}
\newcommand {\bfxi}  {{\bm{\xi}}}

\newcommand {\bfM}  {{M}}

\newcommand {\bfell}  {{\ell}}
\newcommand {\bfPsi}[1]  {{\psi}(#1)}
\newcommand {\bfPhi}[1]  {{\Psi}[#1]}
\newcommand {\bfnu}  {\bm{\nu}}

\newcommand {\F}  {{\mathcal{F}}}
\newcommand {\Fa}  {\mathcal{F}_\alpha}

\def\mod{{\rm mod\,}}

\def\oddots{\mathinner{\mkern1mu\raise\p@ \vbox{\kern7\p@\hbox{.}}\mkern2mu \raise4\p@\hbox{.}\mkern2mu\raise7\p@\hbox{.}\mkern1mu}}

\title{Fractal tiles induced by tent maps}

\author[K.~Scheicher]{Klaus Scheicher{$^1$}}
\address{\tiny $^{1}$
BOKU University, Institute of Mathematics, Department of Natural Sciences and
Sustainable Resources, Gregor Mendel Stra\ss{}e 33, 1180 Vienna, Austria}
\email{klaus.scheicher@boku.ac.at}

\author[V.F.~Sirvent]{V\'{\i}ctor F. Sirvent{$^2$}}
\address{\tiny $^{2}$Departamento de Matem\'aticas, Universidad Cat\'olica del Norte, Antofagasta, Chile}
\email{victor.sirvent@ucn.cl}

\author[P.~Surer]{Paul Surer{$^{3}$}}
\address{\tiny $^{3}$
BOKU University, Institute of Mathematics, Department of Natural Sciences and
Sustainable Resources, Gregor Mendel Stra\ss{}e 33, 1180 Vienna, Austria}
\email{paul.surer@boku.ac.at}

\thanks{The third author's work was supported by the Austrian Science Fund (FWF), 
Project P 28991-N35.}

\date{\today}

\begin{document}

\begin{abstract}
In the present article we deal with geometrical objects induced by the tent maps associated with special Pisot numbers that we call tent-tiles. They are compact subsets of the one-, two-, or three-dimensional Euclidean space, depending on the particular special Pisot number. Most of the tent-tiles have a fractal shape and we study the Hausdorff dimension of their boundary.
Furthermore, we are concerned with tilings induced by tent-tiles.
It turns out that tent-tiles give rise to two types of lattice tilings. 
In order to obtain these results we establish and exploit 
connections between tent-tiles and Rauzy fractals induced by substitutions and automorphisms of the free group. 
\end{abstract}

\maketitle


\begin{section}{Introduction}\label{s:intro}

Consider a real number  $\alpha>1$, let $\beta=\beta(\alpha):=\alpha(\alpha-1)^{-1}$ and define the tent map $T_\alpha$ by
\[T_\alpha: [0,1] \longrightarrow [0,1], x \longmapsto
\begin{cases}
\alpha x, & \text{if }  x \in [0,\alpha^{-1}]; \\ \beta (1-x), & \text{if } x \in [\alpha^{-1},1].
\end{cases}
\]
The dynamics induced by tent maps have been intensively studied in \cite{Lagarias-Porta-Stolarsky:93, Lagarias-Porta-Stolarsky:94} and more recently in  \cite{Scheicher-Sirvent-Surer:16}. 
\begin{figure}[h]
\includegraphics[width=0.13\textwidth]{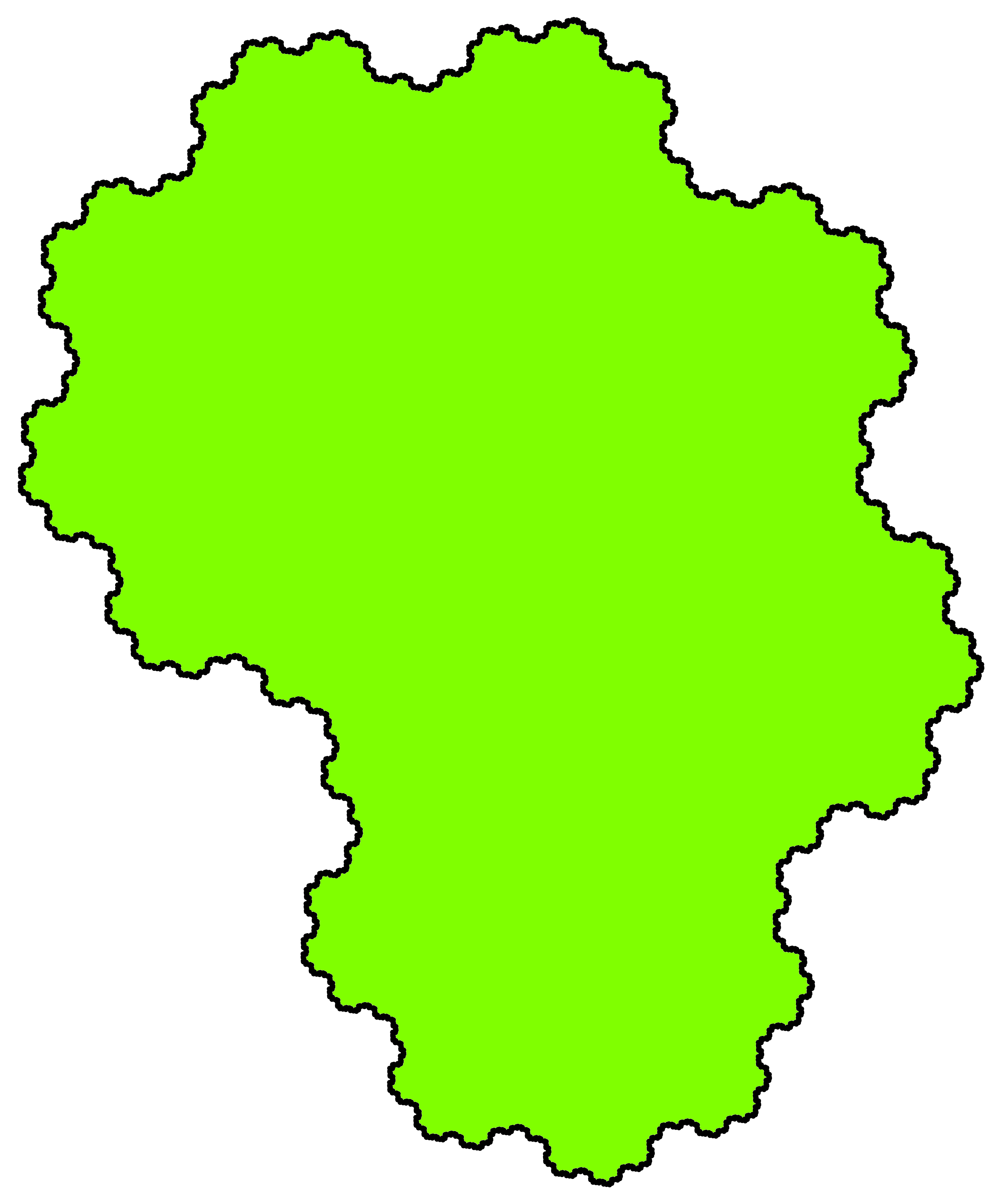}
\hfill
\includegraphics[width=0.15\textwidth]{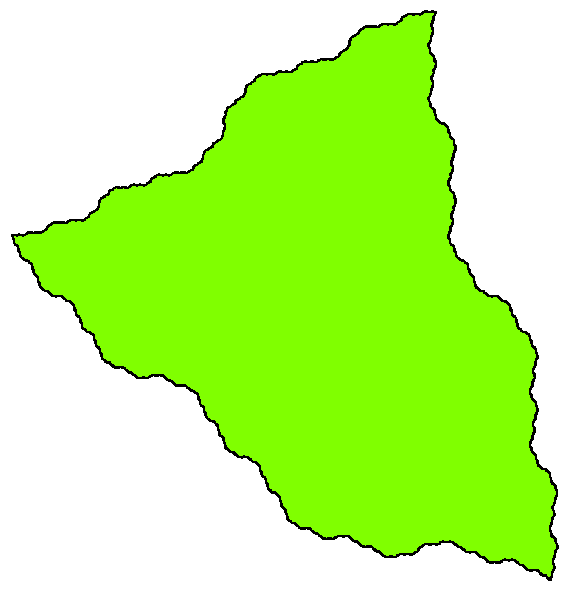}
\hfill
\includegraphics[width=0.15\textwidth]{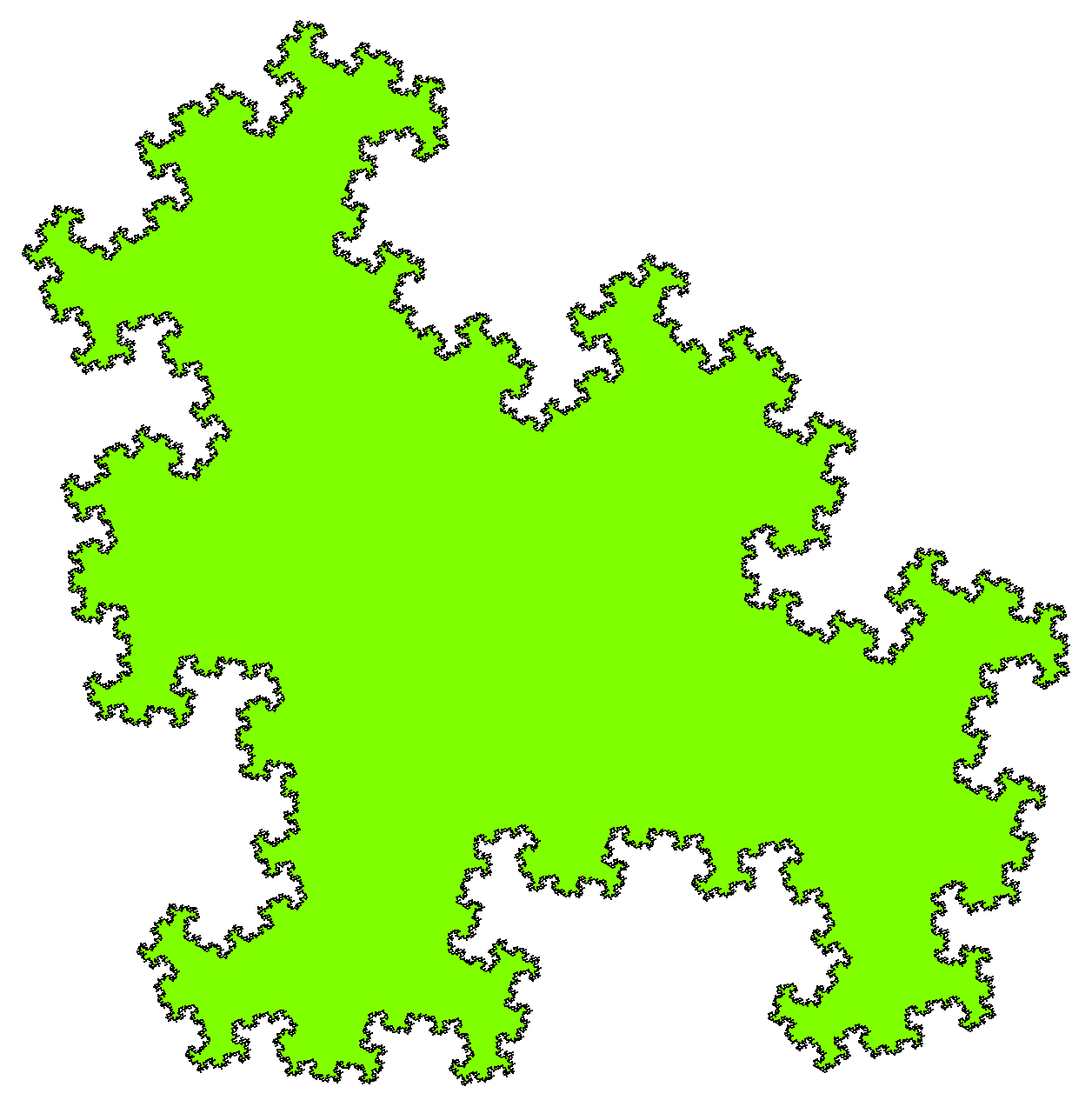}
\hfill
\includegraphics[width=0.13\textwidth]{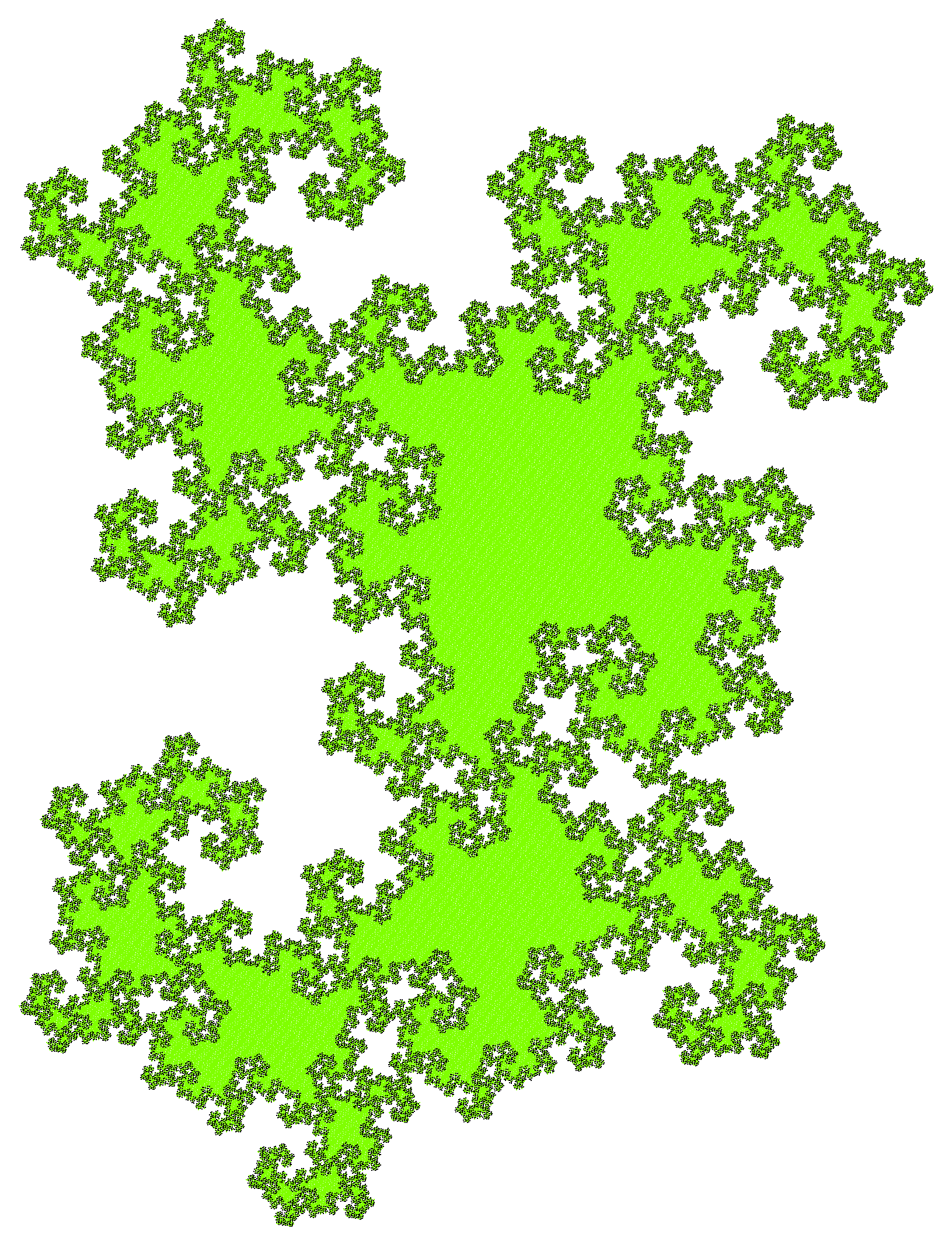}
\hfill
\includegraphics[width=0.15\textwidth]{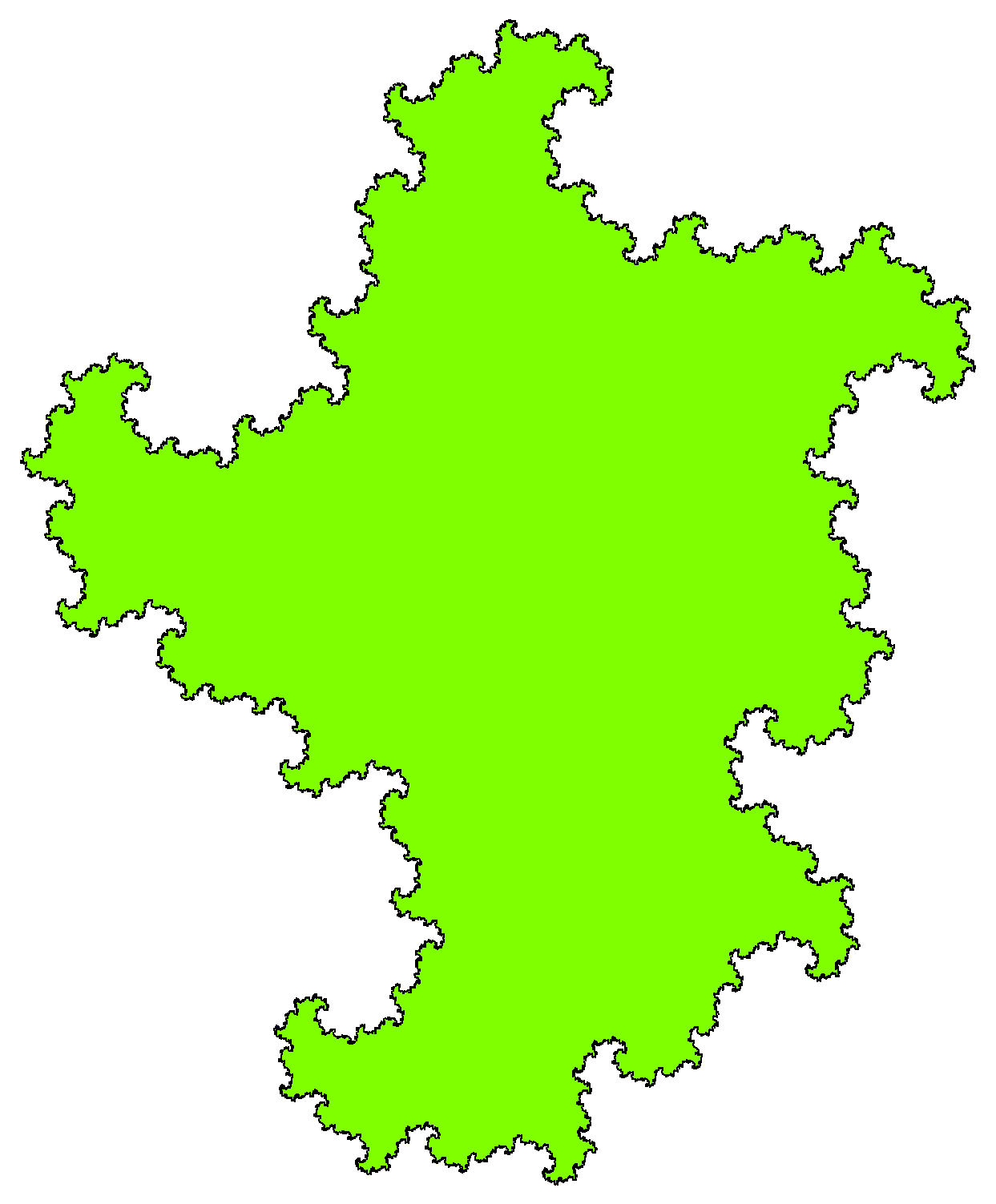}
\hfill
\includegraphics[width=0.14\textwidth]{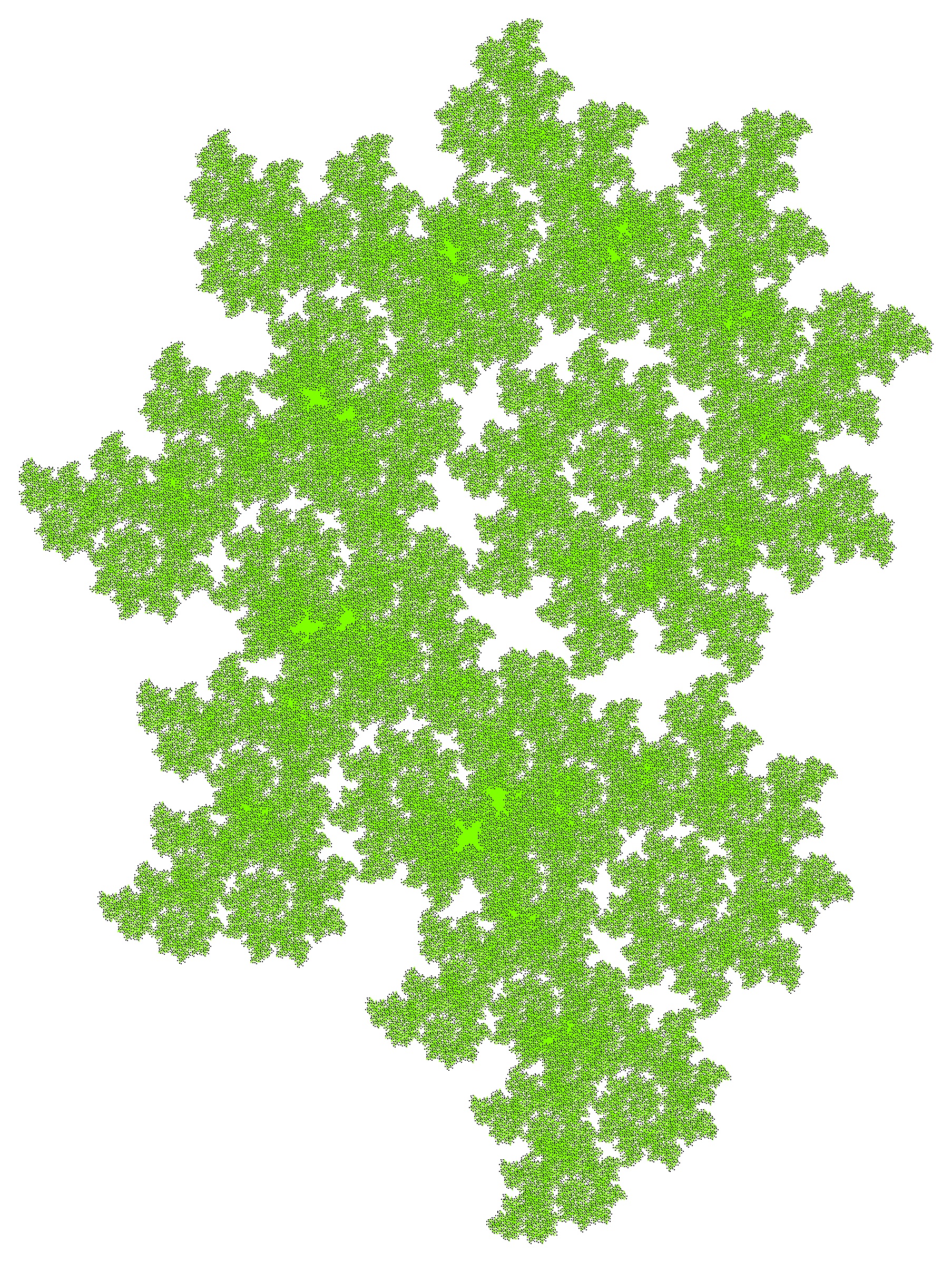}
\caption{There are six planar tent-tiles. They all have a fractal shape.}
\label{planarcases}
\end{figure}

In the present paper we study geometrical objects induced by the tent map that we call {\it tent-tiles}. 
The definition is inspired by the investigation of attractors of tent maps in \cite{Lagarias-Porta-Stolarsky:94}. It is also related to dual systems of algebraic iterated function systems as presented in \cite{Rao-Wen-Yang.2014}. 
The principle idea is to consider the Galois conjugates for algebraic parameters $\alpha$ and to study the respective action of the two branches of $T_\alpha$. In particular,
suppose that $\alpha$ is a Pisot unit such that $\beta$ is a Pisot unit, too. We denote by $d+1$ the algebraic degree of $\alpha$. Let $\alpha^{(1)}, \ldots, \alpha^{(r)}$ be the real Galois conjugates of $\alpha$ and
$\alpha^{(r+1)}, \overline{\alpha^{(r+1)}} \ldots, \alpha^{(r+s)}, \overline{\alpha^{(r+s)}}$ be the non-real complex ones, i.e. $d=r+2s$. We  denoted by $\KK_\alpha$ the Euclidean space
\[\KK_\alpha =\RR^r \times \CC^s \cong \RR^d.\]
For an $z \in \QQ(\alpha)$ and $k \in \{1, \ldots, r+s\}$ we let $z^{(k)} \in \QQ(\alpha^{(k)})$ be the respective Galois conjugate. Define the embedding
\[\psi: \QQ(\alpha) \longrightarrow \KK_\alpha, z \longmapsto (z^{(1)}, \ldots, z^{(r+s)})^T,\]
where $\bm{v}^T$ denotes the  transpose of the vector $\bm{v}$.
Furthermore, for an $z \in \QQ(\alpha)$ we let $\bfPhi{z}$ denote the linear operator
\[\bfPhi{z}: \KK_\alpha \longrightarrow \KK_\alpha, (x_1, \ldots, x_{r+s})  \longmapsto (z^{(1)}\cdot x_1, \ldots, z^{(r+s)}\cdot x_{r+s})^T.\]
Since both $\alpha$ and $\beta$ are Pisot numbers, the two functions
\begin{align*}
\fL: & \KK_\alpha \longrightarrow \KK_\alpha, \bfx \longmapsto \bfPhi{\alpha}(\bfx), \\
\fR: & \KK_\alpha \longrightarrow \KK_\alpha, \bfx \longmapsto \bfPhi{\beta}(\bfPsi{1}-\bfx)
\end{align*}
are contractions and, hence, 
induce an iterated function system in the sense of \cite{Hutchinson.1981} (IFS for short) in the Euclidean space $\KK_\alpha$. The tent-tile is the invariant set induced by this IFS, that is the uniquely determined non-empty compact set $\Fa$ that satisfies $\Fa=\fL(\Fa) \cup \fR(\Fa)$.

A Pisot number $\alpha$ is called a \emph{special Pisot number} when $\beta(\alpha)$ is also a Pisot number. 
Smyth showed in \cite{Smyth.1999} that there exist exactly $11$ special Pisot numbers (see Table~\ref{specialpisotnumbers}) where all of them but $\alpha=\alpha_0=2$ are algebraic units.
Therefore, we only have ten cases to study. However, this allows us to explicitly investigate each of
them. Interestingly, the tiles 
show up many different behaviours and properties. Figure~\ref{planarcases} shows the six planar cases associated with the cubic special Pisot numbers (from left to right) $\alpha_1$, $\alpha_3$, $\alpha_5$, $\alpha_{-1}$, $\alpha_{-3}$, $\alpha_{-5}$.
We see that tent-tiles have a fractal shape.
The special Pisot numbers $\alpha_0=2$ does not induce a tent-tile according to our definition but one could think about a $2$-adic approach. For this reason the respective entry in Table~\ref{specialpisotnumbers} says that $\alpha_0=2$ does not induce an archimedian tent-tile. However, we will not further discuss this idea in the present article.

\begin{figure}[h]
  \begin{minipage}{0.37\textwidth} 
  \includegraphics[width=1.0\textwidth]{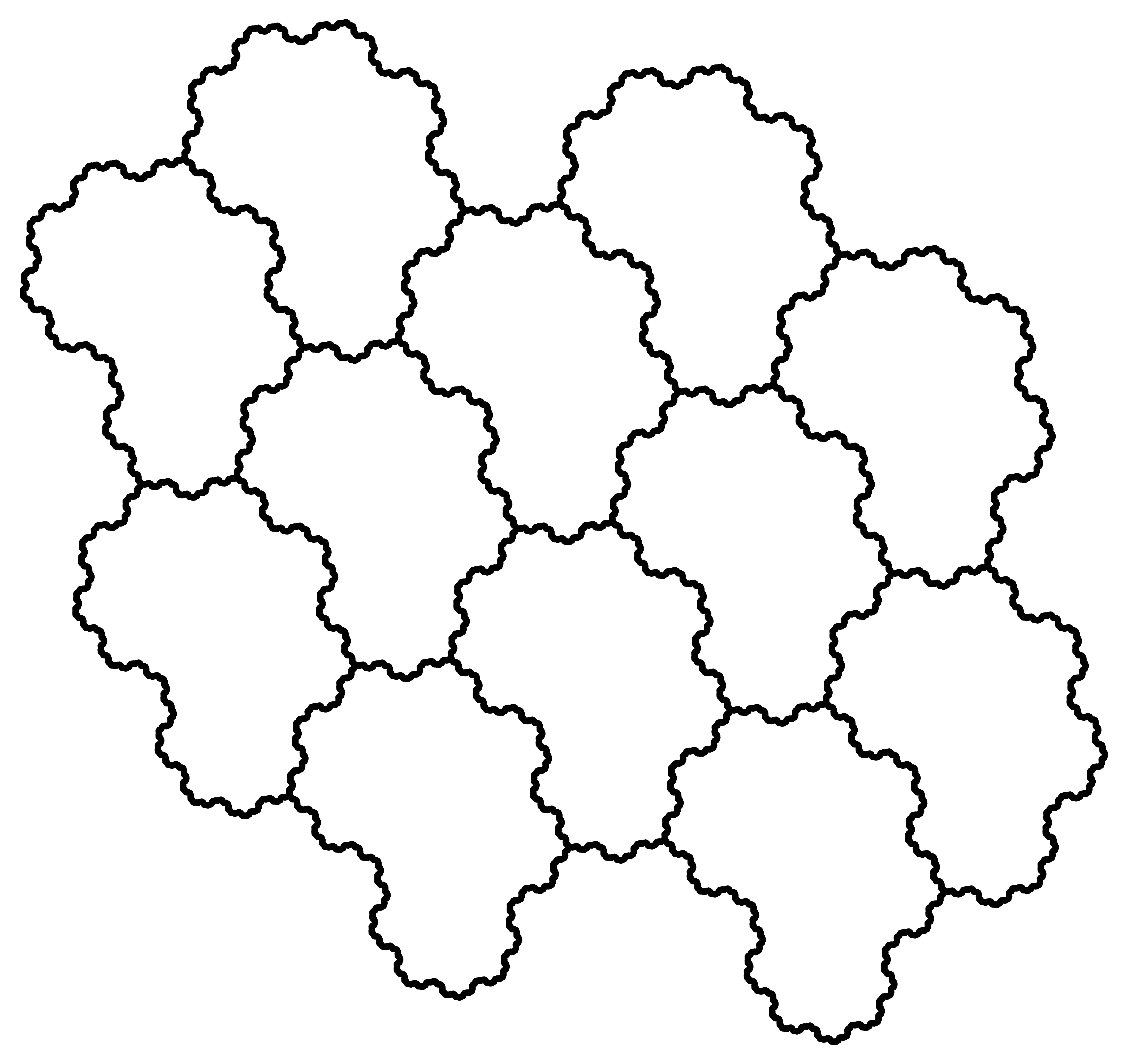}
  \end{minipage}\hfill
  \begin{minipage}{0.48\textwidth} 
  \includegraphics[width=1.0\textwidth]{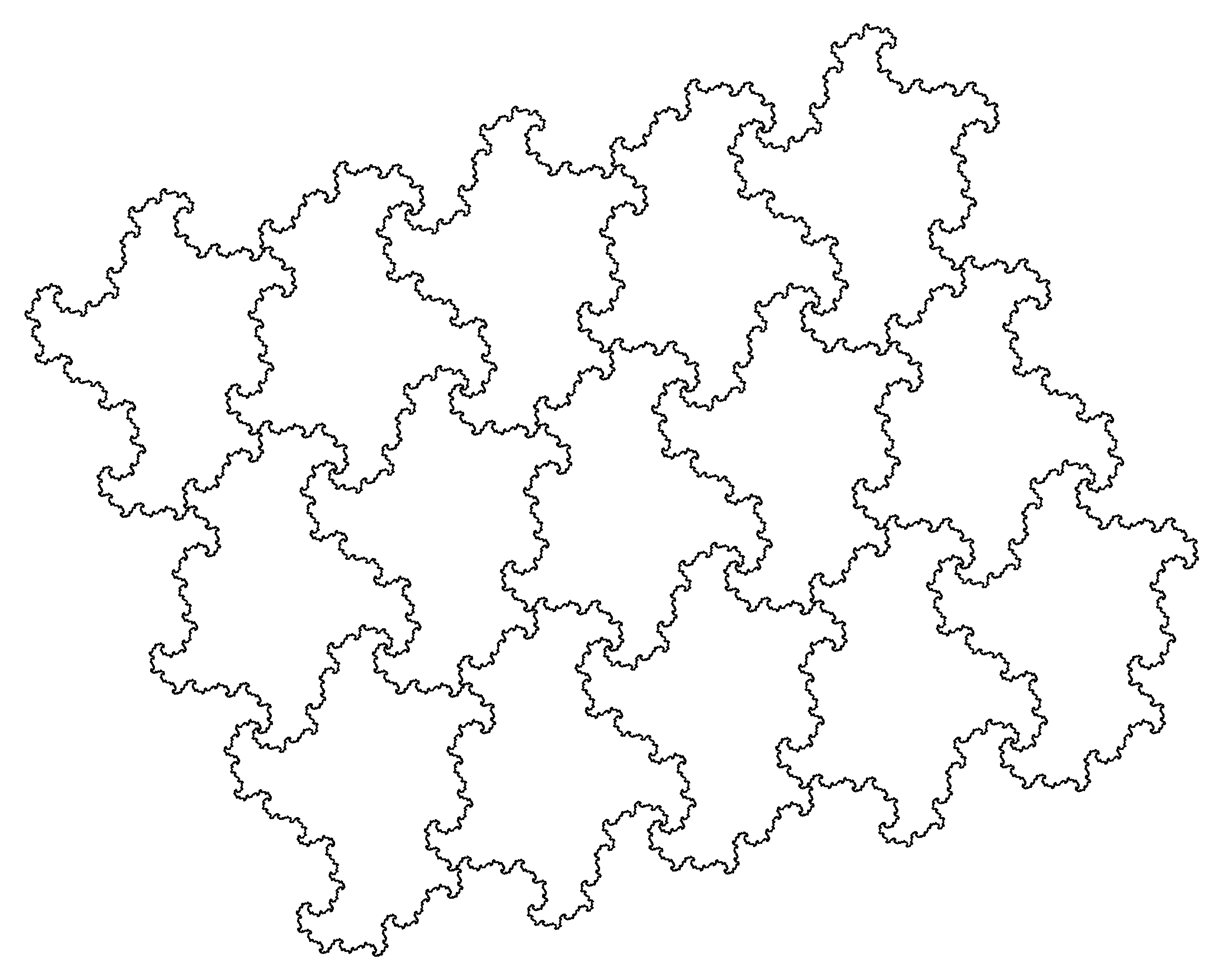}
  \end{minipage}
\caption{Two types of lattice tilings induced by tent-tiles. On the left hand side the tile itself produces the tiling. On the right hand side the tiling is generated by the tile and its reflection.}
\label{tilings}
\end{figure}

As main result we present several topological characteristics of the tent-tiles. Our first result induces that the tiles have positive Lebesgue measure.
\begin{theorem}\label{positivemeasure}
Each tent-tile is a compact sets that is the closure of its interior.
\end{theorem}
We are also interested in the Hausdorff dimension of the boundary of the tent-tiles. Here we obtain either concrete values or at least upper bounds (for the three-dimensional ones) as displayed in Table~\ref{specialpisotnumbers}.
Finally, we are concerned with periodic lattice tilings induced by tent-tiles.
We find out that for some special Pisot numbers the tent-tile provides a periodic lattice tiling of the Euclidean space $\KK_\alpha$ (see the left hand side of Figure~\ref{tilings}). For other special Pisot numbers the associated tent-tile and its reflection produce a periodic tiling (see the right hand side of Figure~\ref{tilings}). 
But there are also two special Pisot numbers that do not seem to provide one of these two types of lattice tiling. The concrete results are also shown in Table~\ref{specialpisotnumbers}.
Summarising, although there exists only a finite number of tent-tiles, they impress by their variety.
Note that IFS (and therefore tent-tiles) also give rise to aperiodic tilings \cite{Barnsley-Vince.2014}. We will not discuss these types of tilings in the present article.

\begin{table}
\begin{tabular}{||l|c|c|c|c||}
\hline
\hline
Number & Value & Minimal polynomial & ${\rm dim}_H(\partial \Fa)$ & Lattice tiling\\
\hline
\hline
$\alpha_{-5}$ & $\approx 4.0796$ & $ t^3-5t^2+4t-1$ & $\approx 1.92089$ & ?\\
\hline
$\alpha_{-4}$ & $\approx 3.62966$ & $ t^4-5t^3+6t^2-4t+1$ & $\leq 2.815$ & yes\\
\hline
$\alpha_{-3}$ & $\approx 3.1479$ & $ t^3-4t^2+3t-1$  & $\approx 1.25074$ & w.r.\\
\hline
$\alpha_{-2}$ & $\approx 2.61803$ & $t^2-3t+1$  & $0$ & yes\\
\hline
$\alpha_{-1}$ & $\approx 2.32472$ & $ t^3-3t^2+2t-1$   & $\approx 1.70018$ & w.r.\\
\hline
$\alpha_{0}$ & $2$ & $ x-2$  & \multicolumn{2}{c||}{No achimedian tent-tile}\\
\hline
$\alpha_1$ & $\approx 1.75488$ & $ t^3-2t^2+t-1$  & $\approx 1.10026$ & yes\\
\hline
$\alpha_{2}$ & $\approx 1.61803$ & $t^2-t-1$  & $0$ & yes\\
\hline
$\alpha_{3}$ & $\approx 1.46557$ & $ t^3-t^2-1$  & $\approx 1.02952$ & w.r.\\
\hline
$\alpha_{4}$ & $\approx 1.38028$ & $ t^4-t^3-1$  & $\leq 2.74421$ & w.r.\\
\hline
$\alpha_{5}$ & $\approx 1.32472$ & $ t^3-t-1$  & $\approx 1.37858$ & ? \\
\hline
\hline
\end{tabular}
\vspace{0.3cm}
\caption{The table shows the $11$ special Pisot numbers. For each $i \in \{-5, \ldots, 5\}$ we have $\beta(\alpha_i)=\alpha_{-i}$. All but $\alpha_0$ are algebraic units and therefore induce tent-tiles.
The penultimate column shows the Hausdorff dimension of the boundary. The last column displays whether the tent-tile induces a lattice tiling, where ``w.r.'' means with reflection, i.e., the lattice tiling also involves the reflected tent-tile. For the two special Pisot numbers marked with ``$?$'' we could not find any type.}\label{specialpisotnumbers}
\end{table}

As the main tool for obtaining our results we show that tent-tiles are closely related with Rauzy fractals.
These fractals are induced by endomorphisms of the free monoid, so called substitutions, and are, technically speaking,
the union of the invariant sets of graph directed iterated function systems (GIFS for short) as discussed, for example, in~\cite{Mauldin-Williams:88}.
This class of fractals is named after Gerard Rauzy who
studied the first such object in 1982 \cite{Rauzy.1982}. Rauzy's approach has been generalised in
\cite{Arnoux-Ito.2001, Ei-Ito-Rao.2006, Ito-Rao.2006}.
Since then Rauzy fractals have been studied in innumerable researches.
An overview can be obtained from more comprehensive discourses and survey articles \cite{Berthe-Siegel-Thuswaldner.2010, Berthe-Siegel.2005, Fogg.2002, Siegel-Thuswaldner.2009}.

One more time we will see how different the particular tent-tiles behave. For our analysis we have to consider four classes of substitutions. We will see that some tent-tiles are given by the induced Rauzy fractals, other tent-tiles are related with the respective Rauzy fractals.  More precisely, the tent-tiles can be obtained by unifying only particular (and sometimes also translated) elements of the invariant set list of the associated GIFS.

Actually, it is only half of the story when we speak about substitutions in context with tent-tiles. In three of the four cases mentioned above there is a much more general theory beyond, namely, dynamics induced by \emph{automorphisms of the free group}.
This theory is a direct extension of substitutions and also induce fractal representations which are,  in fact, Rauzy fractals (with respect to certain substitutions derived from the automorphism).
For a detailed discussion on this topic we refer to \cite{Bestvina-Feighn-Handel.2000,Arnoux-Berthe-Hilion-Siegel.2006}.
The theory is quite complex and a proper introduction would go beyond the scope of this article.
In order to keep our discourse as self-contained as possible we manage to formulate and proof all results in terms of substitutions and
use only
some minor results from \cite{Arnoux-Berthe-Hilion-Siegel.2006} that can be easily understood without knowing the theory beyond.
However, we put several remarks within the paper that contain references and links to the theory of automorphisms of the free group.

As  mentioned above, tent-tiles already have been studied by Lagarias~et.~al.~\cite{Lagarias-Porta-Stolarsky:94} although they were not referred to as ``tent-tiles'' in that article. Our results (specifically Theorems~\ref{positivemeasure}, \ref{tiling1}, \ref{tiling3}, \ref{tiling-1}, \ref{tiling-3}, \ref{TT5} and \ref{TT-5})  give a positive answer to \cite[Conjecture 2.4]{Lagarias-Porta-Stolarsky:94}.
It is also remarkable that the authors mention Rauzy's construction \cite{Rauzy.1982} for comparison. The relation between Rauzy fractals and tent-tiles  that we prove here shows that their intuition was correct. Note that due to a slight programming mistake that occurred at the time, the figures in \cite{Lagarias-Porta-Stolarsky:94} differ from the (correct) ones in the present article (\cite{Lagarias.c2025}). 

The article is organised in the following way. In Section~\ref{sec:preliminaries} we state notations concerning special Pisot numbers and GIFS (which are generalisations of IFS) and use them to 
give a proper definition of the tent-tiles. We also describe what we mean by a tiling. Furthermore, we completely discuss the 
one-dimensional tent-tiles associated with $\alpha_2$ and $\alpha_{-2}$.
In Section~\ref{sec:Rauzy} we introduce the necessary theory of substitutions and Rauzy fractals 
and recall some important results concerning their topological properties. 
In Section~\ref{sec:iwip} we 
state the exact relations between  tent-tiles and Rauzy fractals induced by certain classes of substitutions.
Finally, in Section~\ref{sec:tiling} we use the relations that we elaborated in order to prove our main results.

\end{section}


\begin{section}{Preliminaries}\label{sec:preliminaries}

\subsection{Special Pisot numbers}

Let $\alpha$ be a an algebraic integer of degree $d+1$ and denote by $\alpha^{(1)}, \ldots, \alpha^{(d)}$
the Galois conjugates different from $\alpha$. We say that $\alpha$ is a Pisot number if $\alpha$ is a positive real number and $\abs{\alpha^{(k)}} <1$ for all $k = 1, \ldots, d$. This automatically implies that
$\alpha>1$.

Now, let $\beta=\beta(\alpha):=\alpha(\alpha-1)^{-1}$. Following \cite{Lagarias-Porta-Stolarsky:93} we call $\alpha$ a special Pisot number if both $\alpha$ and $\beta$ are Pisot numbers. Observe that special Pisot numbers appear pairwise. Indeed, from $\beta(\beta(\alpha)))= \alpha$ we conclude that
if $\alpha$ is a special Pisot number then $\beta(\alpha)$ is a special Pisot number, too.

The most obvious special Pisot number is $\alpha=2$. In this case we have $\beta=\alpha$. Smyth~\cite{Smyth.1999} showed that there exist exactly ten further special Pisot numbers that are, in fact, algebraic units.
Table~\ref{specialpisotnumbers} lists the special Pisot numbers and their respective minimal polynomials.
An important observation is that all of these five (unordered) pairs $\{\alpha,\beta\}$ 
are not algebraically independent which means that there exist certain positive integers
$m_\alpha, m_\beta$ such that $\alpha^{m_\alpha} = \beta^{m_\beta}$. The particular dependencies
are collected in Table~\ref{dependencies}. We remark that there exist further algebraic dependencies within the set of special Pisot numbers, for example, $\alpha_5^2=\alpha_1$. However, we will not go into more details since these additional dependencies have no relevance in the present article.

\begin{table}[h]
\begin{tabular}{||l|c|c|c|c|c||}
\hline
\hline
Number & $\alpha_1$ & $\alpha_2$& $\alpha_3$& $\alpha_4$& $\alpha_5$ \\
\hline
$m_i$ & $3$ & $2$ & $3$ & $4$ & $5$\\
\hline
Number & $\alpha_{-1}$ & $\alpha_{-2}$& $\alpha_{-3}$& $\alpha_{-4}$& $\alpha_{-5}$ \\
\hline
$m_i$ & $2$ & $1$ & $1$ & $1$ & $1$ \\
\hline
\hline
\end{tabular}
\vskip 0.2cm
\caption{The table shows the  algebraic dependencies between the pairs of special Pisot numbers. For each $i \in \{-5, \ldots, 5\}$ we have $\beta(\alpha_i)=\alpha_{-i}$ and $\alpha_{i}^{m_i} = \alpha_{-i}^{m_{-i}}$.}\label{dependencies}
\end{table}

\subsection{Graph directed iterated function systems}

Let $\mathbb{E}^d$  be a $d$-dimensional Euclidean space  and denote by $\nm{\cdot}$ the  Euclidean norm. 
A contraction in $\mathbb{E}^d$ is a function $f: \mathbb{E}^d \longrightarrow \mathbb{E}^d$ such that for each $\bfx \not= \bm{0}$ we have $\nm{f(\bfx)} <  \nm{\bfx}$.
A (finite) multigraph $G=G(V,S,E)$ is determined by a finite set of vertices $V$, a set of labels $S$, and 
a set of edges $E \subset V \times S \times V$, where each $(v_1,s,v_2) \in E$ represents an edge from $v$ to $v'$ labelled by $s$. We denote this edge by $v \xrightarrow{s} v'$.

By a realisation of $S$ in $\mathbb{E}^d$ we mean an operator $f$ that assigns to each label $s \in S$ a function
$f_s: \mathbb{E}^d \longrightarrow \mathbb{E}^d$.
A pair $(G,f)$ consisting of a multigraph $G$ and a realisation $f$ is called a
{\it graph-directed iterated function system} (GIFS for short).

\begin{proposition}[{{\it cf}. \cite[Theorem~1]{Bandt.1989b}, \cite[Theorem~1]{Mauldin-Williams:88}}]\label{GIFSpropo}
Let $(G(V,S,E),f)$ be a GIFS such that for each edge $v \xrightarrow{s} v'$ in $G$ the function $f_{s}$ is a contraction.
Then there exists a uniquely determined list $\{X_v: v \in V\}$ of compact sets in $\mathbb{E}^d$ that satisfies 
the system of equations 
\begin{equation}\label{setequations}
X_v = \bigcup_{v \xrightarrow{s} v' \in E} f_s(X_{v'}) \qquad (v \in V).
\end{equation}
The list $\{X_v: v \in V\}$ is frequently referred to as \emph{invariant set list}.
\end{proposition}
The uniqueness of the invariant set list will play an important role in the present article.

Observe that an iterated function system (IFS) in the sense of \cite{Hutchinson.1981} corresponds to  GIFS $(G,p)$ 
where the graph $G$ has only one vertex. Correspondingly, the invariant set list of an IFS consists, in fact, of one invariant set only.

\subsection{Definition of tent-tiles}\label{ssec:tent-tiles}

Consider a special Pisot unit $\alpha$. We denote by $d+1$ the algebraic degree of $\alpha$. Let $\alpha^{(1)}, \ldots, \alpha^{(r)}$ be the real Galois conjugates of $\alpha$ and
$\alpha^{(r+1)}, \overline{\alpha^{(r+1)}} \ldots, \alpha^{(r+s)}, \overline{\alpha^{(r+s)}}$ be the complex ones, i.e. $d=r+2s$.
The  Tent-tiles $\Fa$ associated with $\alpha$ is defined as the set fixed by a graph directed construction in the Euclidean space
\[\KK_\alpha  =\RR^r \times \CC^s \cong \RR^d.\]
We consider the set of labels $\{L,R\}$.
For defining the realisation we let for $z =P(\alpha) \in \QQ(\alpha)$ and $k \in \{1, \ldots, r+s\}$ denote $z^{(k)} =P(\alpha^{(k)}) \in \QQ(\alpha^{(k)})$ the respective Galois conjugate, especially $z^{(k)} \in \RR$ if $k\leq r$  and $z^{(k)} \in \CC$ if $k>r$. Define the embedding
\begin{equation}\label{def:psi}
\psi: \QQ(\alpha) \longrightarrow \KK_\alpha, z \longmapsto (z^{(1)}, \ldots, z^{(r+s)})^T.
\end{equation}
Furthermore, for an $z \in \QQ(\alpha)$ we let $\bfPhi{z}$ denote the linear operator
\[\bfPhi{z}: \KK_\alpha \longrightarrow \KK_\alpha, (x_1, \ldots, x_{r+s})  \longmapsto (z^{(1)}\cdot x_1, \ldots, z^{(r+s)}\cdot x_{r+s})^T.
\]
Note that for all $z_1, z_2 \in \QQ(\beta)$ we have
\begin{align*}
\bfPsi{z_1 z_2} &= \bfPhi{z_1} \left(\bfPsi{z_2}\right), \\
\bfPhi{z_1 z_2} &= \bfPhi{z_1} \circ \bfPhi{z_2}.
\end{align*}
We define the realisation $f$ of $\{L,R\}$ in $\RR^d$ by
\begin{align*}
\fL:& \KK_\alpha \longrightarrow \KK_\alpha,  \bfx \longmapsto \bfPhi{\alpha}(\bfx), \\
\fR:& \KK_\alpha \longrightarrow \KK_\alpha,  \bfx \longmapsto \bfPhi{\beta}(\bfPsi{1}-\bfx).
\end{align*}
Note that we have
\begin{align*}
\fL^{-1}:& \KK_\alpha \longrightarrow \KK_\alpha, \bfx \longmapsto \bfPhi{\alpha^{-1}}(\bfx), \\
\fR^{-1}:& \KK_\alpha \longrightarrow \KK_\alpha, \bfx \longmapsto \bfPsi{1}-\bfPhi{\beta^{-1}}(\bfx).
\end{align*}
Observe  that $\fL$ as well as $\fR$ are contractions since $\alpha$ is a special Pisot number.

\begin{definition}\label{deftenttile}
Let $\alpha$ be a special Pisot number. The {\it tent-tile} $\Fa \subset \KK_\alpha$  associated with $\alpha$ is the invariant set of the
GIFS $(G_1,f)$, where $G_1=G(V,\{L,R\},E)$ with $V:=\{0\}$ and  $E:=\{0 \xrightarrow{L} 0, 0 \xrightarrow{L} 0\}$ (see Figure~\ref{tentgraph}).
\end{definition}
Observe that tent-tiles are uniquely determined up to a linear isomorphism since the embedding $\psi$ depends on the particular ordering of the Galois conjugates of $\alpha$. 
\begin{figure}[ht]
\begin{tikzpicture}

\node[circle,draw=blue!60, very thick, minimum size=10mm] (v0) {$0$};

every edge quotes/.style = {font=\footnotesize, shorten >=1pt},

\draw[->,black, very thick, loop right] (v0) edge  node[pos=0.5, above=0pt] {L} (v0);
\draw[->,black, very thick, loop left] (v0) edge  node[pos=0.5, above=0pt] {R} (v0);

\end{tikzpicture}
\caption{The tent-tiles are given by the invariant set of the GIFS determined by the depicted graph.}
\label{tentgraph}
\end{figure}
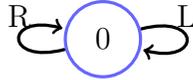

\subsection{Tilings}

We consider a $d$-dimensional Euclidean space $\mathbb{E}^d$. A collection $\{Y_i: i \in I\}$ of compact subsets of $\mathbb{E}^d$ is a multi-tiling of
$\mathbb{E}^d$ if 
\begin{enumerate}
\item $\bigcup_{i \in I} Y_i = \mathbb{E}^d$ (covering property);
\item for each compact subset $Y$ of $\mathbb{E}^d$ the set $\{i \in I: Y \cap Y_i \not= \emptyset\}$ is finite (local finiteness);
\item there exists a positive integer $c$ such that for almost all $\bfx \in \mathbb{E}^d$ (with respect to the $d$-dimensional Lebesgue-measure)
the set $\{i \in I: \bfx \in Y_i\}$ has exactly $c$ elements.
\end{enumerate}
The integer $c$ is the so-called {\it covering degree}. If $c=1$ then we say that the multi-tiling is a 
proper tiling or just a tiling.

\subsection{One dimensional tent-tiles}
We have two special Pisot units whose associated tent-tile is a subset of the real line: $\alpha_{-2}$ and $\alpha_{2}$. For these two cases the situation is quite clear and can be shown directly.
\begin{theorem}[{cf.~\cite[Corollary~2.3a]{Lagarias-Porta-Stolarsky:94}}]\label{MT2-2}
Let $\phi=\frac{1+\sqrt{5}}{2}$ be the golden ratio.
Then
\begin{align*}
\F_{\alpha_{-2}} & =[-\phi,0] \in \KK_{\alpha_{-2}} = \RR, \\
\F_{\alpha_{2}} & =[\nicefrac{(1-\phi)}{2},\nicefrac{1}{2}]\in \KK_{\alpha_{2}} = \RR.
\end{align*}
\end{theorem}
\begin{proof}
Due to the uniqueness of the invariant set list it suffices to show that
the stated intervals satisfy the set equation
induced by the  GIFS $G_1$ from Definition~\ref{deftenttile}.
\end{proof}

We immediately see that $\F_{\alpha_{-2}}$ and $\F_{\alpha_{2}}$ have a positive length (i.e., positive $1$-dimensional Lebesgue measure) and ${\rm dim}_H(\partial \F_{\alpha_{-2}})=0$ as well as ${\rm dim}_H(\partial \F_{\alpha_{2}})=0$. Furthermore, the collections $\{u_1\phi+\F_{\alpha_{-2}} : u_1 \in \ZZ\}$ and $\{u_1\frac{\phi}{2}+\F_{\alpha_{2}} : u_1 \in \ZZ\}$ obviously tile the real line in a periodic way.

For higher dimensional tent-tiles we need a more sophisticated strategy to prove our results. 

\end{section}


\begin{section}{Rauzy fractals}\label{sec:Rauzy}

In this section  we properly define Rauzy fractals and present some of their properties that we need in order to prove our results.
There exist several (equivalent) definitions. In the actual one we define Rauzy fractals as invariant sets of particular IFS where we  makes use of the formalism that we introduced in context with tent-tiles.  This approach will turn out to be the most convenient one in our further proceeding.

Let $\A=\{0, \ldots, m\}$ be an alphabet and $\A^*$ the set of finite words over $\A$,
which includes the empty word $\varepsilon$ too.
For a word $W \in \A^*$ and $a \in \A$ we let $\abs{W}_a$ denote the number of occurrences of $a$ in $W$. Furthermore, we define 
$\bfell: \A^* \longmapsto \RR^m$ by
\[\bfell(W) = (\abs{W}_0, \abs{W}_1, \ldots, \abs{W}_m)^T.\]

A substitution over $\A$ is a non-erasing morphism $\zeta: \A^* \longmapsto \A^*$. We say that $\zeta$ is primitive if there exists a power $k$ such that for each $a \in \A$ all entries of $\bfell(\zeta^k(a))$ are strictly positive.
Define the incidence matrix $\bfM_\zeta := (\bfell(\zeta(0)), \bfell(\zeta(1)), \ldots, \bfell(\zeta(m))) \in \RR^{(m+1)\times (m+1)}$ and observe that for each $W \in \A^*$ we have $\bfell(\zeta(W))= \bfM_\zeta \, \bfell(W)$.

We define the prefix-graph $\Gamma_\zeta$ in the following way. Its set of vertices is $V=V(\Gamma_\zeta):=\{0, \ldots, m\}$. The labels are the proper prefixes induced by the substitution, therefore 
$S:=\A^*$.
There is an edge from $a$ to $b$ labelled by $U$ if $\zeta(b)=UaV$ for a word $V \in \A^*$, hence
\[E=E(\Gamma_\zeta):=\left\{a \xrightarrow{U} b: \,  \exists V \in \A^*\text{ such that } \zeta(b)=UaV\right\}.\]
The matrix $\bfM_\zeta$ is the adjacency matrix of $\Gamma_\zeta$.

If $\zeta$ is a primitive substitution then $\Gamma_\zeta$ is strongly connected and $\bfM_\zeta$ is a primitive matrix. Due to the Perron-Frobenius theorem it possesses a dominant positive (real) eigenvalue
that we denote by $\lambda$. Let $d+1$ be the algebraic degree of $\lambda$ and $\lambda^{(1)}, \ldots, \lambda^{(d)}$ the Galois conjugates of $\lambda$. 
Depending on the shape of $\bf{M}_\zeta$ and its eigenvalues we define several attributes for the primitive substitution $\zeta$:
\begin{itemize}
\item $\zeta$ is an {\it irreducible} substitution if $m=d$;
\item $\zeta$ is {\it unimodular} if ${\rm det}(\bfM_\zeta)=\pm 1$;
\item $\zeta$ is a {\it Pisot substitution} if $\alpha$ is a Pisot number, that is $\max_{k \in \{1, \ldots, d+1\}} \abs{\lambda^{(k)}}<1$.
\end{itemize}
For our further proceeding we exclusively concentrate on unimodular Pisot substitutions. 

We let $\bfnu \in \QQ(\alpha)^{m+1}$ denote a left eigenvector of $\bfM_\zeta$ with respect to the dominant eigenvalue
$\lambda$. By Perron-Frobenius theorem we may assume that all entries of $\bfnu$ are strictly positive.

Assume that $\lambda^{(1)}, \ldots, \lambda^{(r)} \in \RR$ and
$\lambda^{(r+1)}, \overline{\lambda^{(r+1)}} \ldots, \lambda^{(r+s)}, \overline{\lambda^{(r+s)}} \in \CC\setminus \RR$ and set
\[\KK_\zeta = \RR^r \times \CC^s.\]
We consider 
the embedding
\[\pi: \QQ(\lambda) \longrightarrow \KK_\zeta, z \longmapsto (z^{(1)}, \ldots, z^{(r+s)})^T,\]
where $z^{(k)} \in \QQ(\lambda^{(k)})$ is the respective Galois conjugate of $z$ for each $k \in \{1, \ldots, r+s\}$.
Furthermore, let $h$ be the linear map
\[h: \KK_\zeta \longrightarrow \KK_\zeta, (x_1, \ldots, x_{r+s})  \longmapsto (\lambda^{(1)}\cdot x_1, \ldots, \lambda^{(r+s)}\cdot x_{r+s})^T.
\]
Obviously the fact that $\lambda$ is a Pisot unit implies $h$ to be a contraction
on $\KK_\zeta$. Note that for all $x \in \QQ(\lambda)$ we have $\pi(\lambda \cdot x)=h\circ \pi(x)$.

Now, the Rauzy fractal associated with the primitive unimodular Pisot substitution $\zeta$ is given by the GIFS $(\Gamma_\zeta,g)$ where the realisation $g$
assigns to each prefix $U \in S = \A^*$ the function
\[g_U: \KK_\zeta \longrightarrow \KK_\zeta, \bfx \longmapsto \pi\left(\spk{\bfnu,\bfell(U)}\right) + h(\bfx),\]
where $\spk{\cdot,\cdot}$ denotes the inner product.
Obviously $g_U$ is a contraction for any choice of $U$. Therefore, the GIFS $(\Gamma_\zeta,g)$ possesses an 
invariant set list that we denote by  $\R_\zeta(0), \ldots \R_\zeta(m) \subset \KK_\zeta$. The Rauzy fractal $\R_\zeta$ is the union
\begin{equation}\label{Rauzyunion}
\R_\zeta = \R_\zeta(0) \cup \ldots \cup \R_\zeta(m).
\end{equation}
Note that the Rauzy fractal $\R_\zeta$ is uniquely determined up to linear transformation that depends on the ordering of the Galois conjugates of the dominant root $\lambda$ and the actual choice of the eigenvector $\bfnu$.

\begin{remark}
Observe that the notions $\KK_\zeta$, $\pi$ and $h$ are completely analogous to  $\KK_\alpha$, $\psi$ and $\Psi[\alpha]$, respectively, that appear in context with the definition of tent-tiles.
The different way of notation has been chosen in order to highlight the different origins--tent-tiles are induced by a special Pisot unit $\alpha$, therefore the respective space is called $\KK_\alpha$, and  Rauzy fractals are induced by substitution $\zeta$, consequently we call the space $\KK_\zeta$. Eventually, tent-tiles and Rauzy fractals are different objects and we clearly separate them. In this way we also avoid to  constantly specify the dominant root of a substitution. 
However, later, in Section~\ref{sec:iwip}, we use the analogies in order to show equivalences between tent-tiles and Rauzy fractals induced by certain classes of substitutions.
\end{remark}

\begin{proposition}[{\cite[Theorem~4.1]{Sirvent-Wang.2002}}]\label{Rauzypositivemeasure}
Let $\zeta$ be a primitive unimodular substitution over the alphabet $\A$. Then for each letter $a \in \A$ the set $\R_\zeta(a)$ is
the closure of its interior.
\end{proposition}

The union \eqref{Rauzyunion} is not necessarily disjoint with respect to the $d$-dimensional Lebesgue measure.

We say that two letters $a_1$ and $a_2$ have {\it strong coincidence at prefixes}
if  there exists an integer $k \geq 1$, a letter $b \in \A$, and words $P_1, P_2, S_1, S_2 \in \A^*$ such that $\zeta^k(a_1)=P_1bS_1$, $\zeta^k(a_2)=P_2bS_2$ and $\bfell(P_1)=\bfell(P_2)$.

We say that two letters $a_1$ and $a_2$ have {\it weak coincidence at prefixes}.
if  there exists an integer $k \geq 1$, a letter $b \in \A$, and words $P_1, P_2, S_1, S_2 \in \A^*$ such that $\zeta^k(a_1)=P_1bS_1$, $\zeta^k(a_2)=P_2bS_2$ and $\spk{\bfnu,\bfell(P_1)}= \spk{\bfnu,\bfell(P_2)}$.

\begin{definition}[Coincidence condition]
A primitive unimodular Pisot substitution $\zeta$ over $\A$ satisfies the {\it strong coincidence condition} if
all letters have strong coincidence at prefixes, and $\zeta$  satisfies the {\it weak coincidence condition} if all letters have weak coincidence at prefixes.
\end{definition}
Observe that for irreducible substitutions the notions strong coincidence and weak coincidence are equivalent. 
It is conjectured that all  primitive irreducible unimodular Pisot substitution satisfy the strong coincidence condition (coincidence conjecture). Up to now this is only confirmed in the two-letter case \cite{Barge-Diamond.2002}.

The strong coincidence condition implies the union \eqref{Rauzyunion} to be (measure-theoretically) disjoint (see ~\cite{Arnoux-Ito.2001, Canterini-Siegel.2001b, Ei-Ito-Rao.2006, Ito-Rao.2006, Sirvent-Wang.2002}).  We show that the weak coincidence is also sufficient.
\begin{lemma}\label{coinc}
Let $\zeta$ be a primitive unimodular Pisot substitution and suppose that $a_1, a_2 \in \A$ have weak coincidence at prefixes. Then
$\R_\zeta(a_1)$ and $\R_\zeta(a_2)$ are measure-theoretically disjoint.
\end{lemma}
\begin{proof}
Observe that for all integers  $k \geq 1$ the substitution $\zeta^k$
induces an isomorphic Rauzy fractal as $\zeta$. Indeed, the prefix graph $\Gamma_{\zeta^k}$ corresponds to the graph obtained from $\Gamma_{\zeta}$ by
considering the paths of length $k$, the $M_{\zeta}$ and $M_{\zeta^k}$ have the same eigenspaces, and since $\QQ(\beta)=\QQ(\beta^k)$ we have $\KK_\zeta \cong \KK_{\zeta^k}$. Therefore we may assume, without loss of generality, 
that $\zeta(a_1)=P_1bS_1$ and  $\zeta(a_2)=P_2bS_2$ with $\spk{\bfnu,\bfell(P_1)}=\spk{\bfnu,\bfell(P_2)}$.
This implies that $(b,P_1,a_1)$ and $(b,P_2,a_2)$ are edges of the prefix-graph $\Gamma_\zeta$.
By definition, $\R_\zeta(b)$ satisfies the GIFS-equation
\[\R_\zeta(b) = \bigcup_{b \xrightarrow{U} a \in E(\Gamma_\zeta)} g_U(\R_\zeta(a)).\]
This union is measure-theoretically disjoint (see the proof of \cite[Theorem 2]{Berthe-Siegel.2005}) and
by the definition of the realisation $g$ this implies  $\R_\zeta(a_1)$ and $\R_\zeta(a_2)$ to be measure-theoretically disjoint.
\end{proof}

\begin{proposition}[{{\it cf}~\cite[Proposition 3.15]{Siegel-Thuswaldner.2009}}]\label{Rauzytiling}
Let $\zeta$ be an unimodular Pisot substitution such that the elements of the invariant set list $\{\R_\zeta(a): a \in \A\}$ have pairwise disjoint interior. Denote by $d+1$ the algebraic degree of the dominant eigenvalue and by $\bfnu$
an associated positive left eigenvector.
If there exist distinct letters $a_1, \ldots, a_{d+1} \in \A$ such that
\begin{multline}\label{qmc} 
\forall a \in \A: \pi\left(\spk{\bfnu,\bfell(a) - \bfell(a_{d+1})}\right) \in \\ \Xi_{\rm lat}:=\left\{\sum_{i=1}^d u_i \pi\left(\spk{\bfnu,\bfell(a_i) - \bfell(a_{d+1})}\right): u_1, \ldots, u_d \in \ZZ\right\}
\end{multline}
then the collection $\left\{\bfx + \R_\zeta: \bfx \in \Xi_{\rm lat}\right\}$ 
is a multi-tiling of the  space $\KK_\zeta$.
\end{proposition}
Note that Condition~\eqref{qmc} is known as {\it quotient mapping condition} and is obviously satisfied if $\zeta$ is an irreducible substitution.

We refer to the above collection as the {\it lattice multi-tiling induced by $\zeta$} and we say that $\zeta$ possesses the tiling property if the induced lattice multi-tiling is a proper tiling. Observe that by the so-called Pisot conjecture irreducible (unimodular Pisot) substitutions always have the tiling property. This conjecture, which is one of the most famous open problems in context with Rauzy fractals, is extensively discussed in \cite{Akiyama-Barge-Berthe-Lee-Siegel:15}. Up to now it is confirmed for two letter substitutions \cite{Barge-Diamond.2002, Hollander-Solomyak.2003} and a class of substitutions related with beta-expansions \cite{Barge.2018}.

In order to show our results concerning the topology of tent-tiles we will use some tools elaborated in \cite{Siegel-Thuswaldner.2009}. We quickly state the most important facts.

At first we introduce another type of multi-tiling induced by the Rauzy fractal $\R_\zeta$, more precisely, by the sets  $\R_\zeta(0), \ldots, \R_\zeta(m)$. This multi-tiling is aperiodic and it is called the {\it self-replicated tiling}.

\begin{proposition}[{{\it cf}~\cite[Proposition 3.7]{Siegel-Thuswaldner.2009}}]\label{SRtiling}
Let $\zeta$ be an unimodular Pisot substitution over the alphabet $\A=\{0,\ldots, m\}$, denote by $\bfnu$
a positive left eigenvector associated  with the dominant eigenvalue, and
define the self-replicated translation set
\[\Xi_{\rm sr}:=\left\{\left(\pi\left(\spk{\bfnu,\bfxi}\right),a\right)
 \in \KK_\zeta \times \A: \bfxi \in \ZZ^{m+1},  0 \leq \spk{\bfnu,\bfxi} <\spk{\bfnu,\bfell(a)}\right\}.\]
Then the collection $\left\{\bfx + \R_\zeta(a): (\bfx,a) \in \Xi_{{\rm sr}}\right\}$
is a multi-tiling (the self-replicating multi-tiling) of the space $\KK_\zeta$.
\end{proposition}
Note that for irreducible substitutions the self-replicating multi-tiling is a proper tiling if and only if the lattice multi-tiling is proper (see \cite{Ito-Rao.2006}). An analogous statement for the reducible case (provided that Condition~\eqref{qmc} is satisfied)  does not seem to exist but
due to the lack of counterexamples one may conjecture that such an equivalence holds for reducible substitutions, too.
However, in the present article the covering degree of the self-replicating multi-tiling is not important.

We now introduce two graphs induced by the two types of multi-tilings, the self-replicating boundary graph
$\Gamma_\zeta^{({\rm sr})}$ and the lattice boundary graph $\Gamma_\zeta^{({\rm lat})}$. Originally these graphs are defined in \cite[Section~5.2]{Siegel-Thuswaldner.2009} where the interested reader finds a more general discussion. We adopt here the more compact definition given in \cite{Loridant-Messaoudi-Surer-Thuswaldner.2013}.

The vertices of the two boundary graph are contained in the set
\[\mathcal{D}:=\left\{[a_1,\pi\left(\spk{\bfnu,\bfxi}\right),a_2] \in  \A\times\KK_\zeta\times \A: \begin{array}{l} \bfxi \in \ZZ^{m+1},  \\(0 < \spk{\bfnu, \bfxi}) \text{ or } 
(\spk{\bfnu, \bfxi}= 0 \text{ and } a_1 < a_2)\end{array}\right\}.\]
Recall that $m+1=\abs{\A}$ is the size of the alphabet and that $\bfnu$ is a positive left eigenvector associated with the dominant eigenvalue
of $\zeta$. The labels are elements of $\KK_\zeta$.
\begin{definition}[{{\it cf}.~\cite[Definition~2.1]{Loridant-Messaoudi-Surer-Thuswaldner.2013}}]\label{def:srsG}
The \emph{self-replicating boundary graph} $\Gamma_\zeta^{({\rm sr})}$ and the \emph{lattice boundary graph} $\Gamma_\zeta^{({\rm lat})}$, respectively, are the largest graphs with the following properties
\begin{enumerate}
\item The vertices $[a_1,\bfx,a_2]$ are elements of $\mathcal{D}$ such that
\[\nm{\bfx} \leq \frac{2\max_{(a,U,b) \in E(\Gamma_\zeta)} \nm{\pi(\spk{\bfnu, \bfell(U)}) }}{1-\max_{k \in \{1, \ldots, r+s\}}\abs{\lambda^{(k)}}}.\]
\item There is an edge $e \xrightarrow{s} e'$ with $e = [a_1,\bfx,a_2]$, $e'=[b_1,\bfx',b_2]$ if and only
if either
\begin{itemize}
\item there exist
$a_1 \xrightarrow{U_1} b_1, a_2 \xrightarrow{U_2} b_2 \in E(\Gamma_\zeta)$ such that $h(\bfx')=\bfx' + \pi(\spk{\bfnu,\bfell(U_2)-\bfell(U_1)})$ or
\item there exist
$a_1 \xrightarrow{U_1} b_2, a_2 \xrightarrow{U_2} b_1 \in E(\Gamma_\zeta)$ such that 
$h(\bfx')=-\bfx' + \pi(\spk{\bfnu, \bfell(U_1)-\bfell(U_2)})$.
\end{itemize}
The respective label is given by
\[s=\begin{cases}\pi(\spk{\bfnu, \bfell(U_1)}), & 
\text{if } \spk{\bfnu, \bfell(U_1)} \leq \spk{\bfnu, \bfell(U_2) + \bfxi}; \\
\pi(\spk{\bfnu, \bfell(U_2)}) + \bfx, & \text{otherwise}. \end{cases}\]
where $\bfxi \in \ZZ^{m+1}$ such that $\pi(\spk{\bfnu, \bfxi}) = \bfx$.

\item Each vertex belongs to an infinite walk that starts from a vertex $[a_1,\bfx,a_2]$ with  $(\bfx,a_2) \in \Xi_{sr}$ (in the case of $\Gamma_\zeta^{({\rm sr})}$) or $\bfx \in \Xi_{lat}\setminus\{\mathbf{0}\}$ (in the case of $\Gamma_\zeta^{({\rm lat})}$).
\end{enumerate}
\end{definition}
Observe that by \cite[Proposition~5.5]{Siegel-Thuswaldner.2009} both $\Gamma_\zeta^{({\rm lat})}$ and $\Gamma_\zeta^{({\rm sr})}$ are well defined and finite. In accordance with our previous way of notation we let
$V(\Gamma_\zeta^{({\rm lat})})$ denote the set of vertices of the graph $\Gamma_\zeta^{({\rm lat})}$ and $E(\Gamma_\zeta^{({\rm lat})})$ its set of edges. Analogously, $V(\Gamma_\zeta^{({\rm sr})})$ and  $E(\Gamma_\zeta^{({\rm sr})})$ are the vertices and edges, respectively, of $\Gamma_\zeta^{({\rm sr})}$.

The lattice boundary graph and the self-replicated boundary graph contain valuable informations concerning the tiling property and the topology of Rauzy fractals. We let $\mu_{\rm lat}$ denote the dominant eigenvalue of the incidence matrix of the lattice boundary graph $\Gamma_\zeta^{({\rm lat})}$. Analogously, $\mu_{\rm sr}$ is the dominant eigenvalue of the incidence matrix of the self-replicated boundary graph $\Gamma_\zeta^{({\rm sr})}$. 
At first we state that these eigenvalues  allow us to verify whether the respective multi-tilings are proper.
\begin{proposition}[{{\it cf}.~\cite[Theorem~4.1]{Siegel-Thuswaldner.2009}}]\label{tilingproperty}
Let $\zeta$ be a unimodular Pisot substitution. Then the self-replicating multi-tiling is a proper tiling if and only if $\mu_{\rm sr} < \lambda$. If $\zeta$ satisfies Condition~\eqref{qmc} then the lattice multi-tiling  is a proper tiling if and only if $\mu_{\rm lat} < \lambda$. 
\end{proposition}

For studying the dimension of the boundary of Rauzy fractals we concentrate on the self-replicating boundary graph.
This graph allows us to describe the intersections of the elements of the self-replicating multi-tiling  in terms of a GIFS. In particular, we define the realisation $g'$ that assigns to each $\bfy \in \KK_\zeta$ the linear map
\[g'_{\bfy}: \KK_\zeta \longrightarrow \KK_\zeta, \bfx \longmapsto \bfy+h(\bfx),\]
which is a contraction since $h$ is a contraction. 
The GIFS $(\Gamma_\zeta^{({\rm sr})}, g')$ induces an invariant set list $\{X_v: v \in V(\Gamma_\zeta^{({\rm sr})})\}$.
We have the following lemma.
\begin{lemma}[{cf.~\cite[Theorem 5.7]{Siegel-Thuswaldner.2009}}]\label{limm1}
Let $\zeta$ be a primitive unimodular Pisot substitution over the alphabet $\A$.
Then $[a_1, \bfx, a_2] \in V(\Gamma_\zeta^{({\rm sr})})$ if and only if $(\bfx,a_2) \in \Xi_{\rm sr}$ and $\R_\zeta(a_1) \cap (\bfx + \R_\zeta(a_2)) \not=\emptyset$.
In this case we have $X_{[a_1, \bfx, a_2]}=\R_\zeta(a_1) \cap  \bfx + \R_\zeta(a_2)$.
\end{lemma}

The next result is our main tool in context with fractal dimensions. We first consider the box counting dimension $\dim_B$ as an upper bound for the Hausdorff dimension. The proposition is based on
\cite[Theorem 4.4]{Siegel-Thuswaldner.2009} but we do not require that the self-replicating boundary graph is strongly connected. In return, we get a slightly weaker statement.
\begin{proposition}\label{Bdim}
Let $\zeta$  be a primitive unimodular Pisot substitution over the alphabet $\A$ and suppose that
$\mu_{\rm sr}<  \lambda$. Then there exists an $a \in \A$ such that
\[\dim_B(\partial \R_\zeta(a)) = d  + \frac{ \log(\lambda) - \log(\mu_{\rm sr})}{\min_{k \in \{1, \ldots, r+s\}}\log(\abs{\lambda^{(k)}})}.\]
\end{proposition}
\begin{proof}
Let $\mathcal{C}$ be the set of strongly connected components of $\Gamma_\zeta^{({\rm sr})}$. For each $C \in \mathcal{C}$ we
denote by $\mu_C$ the dominant eigenvalue of the adjacency matrix of $C$. Clearly, $\mu_{\rm sr}=\max_{C \in \mathcal{C}}\mu_C$.

For vertices $v, v' \in V(\Gamma_\zeta^{({\rm sr})})$
let ${\rm dist}(v,v')$ be the graph distance of $v$ and $v'$, i.e. the length of the shortest path from $v$ to $v'$. If there is no such path then we set ${\rm dist}(v,v'):=\infty$. For a strongly connected component $C \in \mathcal{C}$ we define
${\rm dist}(v,C)=\min\{{\rm dist}(v,v'): v' \text{ is a vertex of } C\}$. 
By observing 
\cite[Theorem 5]{Mauldin-Williams:88} and 
the argumentation of the proof of \cite[Theorem 4.4]{Siegel-Thuswaldner.2009} we see that for each
$v \in V(\Gamma_\zeta^{({\rm sr})})$ we have
\[\dim_B(X_v) = d  + \frac{\log(\lambda) - \max\{\log(\mu_C) : C \in \mathcal{C}, {\rm dist}(v,C)<\infty\}}{\min_{k \in \{1, \ldots, r+s\}}\log(\abs{\lambda^{(k)}})}.\]

Finally, observe that Lemma~\ref{limm1} implies that for each $a \in \A$ we have
\[
\partial{\R_\zeta(a)}=  \bigcup_{\substack{[a_1, \bfx, a_2] \in V(\Gamma_\zeta^{({\rm sr})})\\ a_1=a \vee
 (\bfx=\bm{0} \wedge a_2=a)}} X_{[a_1, \bfx, a_2]},
\]
hence, 
\[
\dim_B(\partial \R_\zeta(a))=  \max\{\dim_B(X_{[a_1, \bfx, a_2]}): [a_1, \bfx, a_2] \in V(\Gamma_\zeta^{({\rm sr})}), a_1=a \vee
 (\bfx=\bm{0} \wedge a_2=a)\}.
\]
From this the statement of the proposition follows immediately.
\end{proof}

Concerning the Hausdorff dimension we obtain the following corollary.
\begin{corollary}\label{Hdim}
Let $\zeta$  be a primitive unimodular Pisot substitution over the alphabet $\A$ and suppose that
$\mu_{\rm sr}< \lambda_0$. Then there exists an $a \in \A$ such that
\[\dim_H(\partial \R_\zeta(a)) \leq d  + \frac{\log(\lambda) - \log(\mu_{\rm sr})}{\min_{k \in \{1, \ldots, r+s\}}\log(\abs{\lambda^{(k)}})}.\]
If $\abs{\lambda^{(1)}}=\abs{\lambda^{(2)}}= \cdots = \abs{\lambda^{(r+s)}}$ then equality holds.
\end{corollary}
\begin{proof}
If $\abs{\lambda^{(1)}}=\abs{\lambda^{(2)}}= \cdots = \abs{\lambda^{(r+s)}}$ then $g'_{\bfy}$ is a similarity map for each $\bfy \in \KK_\zeta$ and we can apply \cite[Theorem 5]{Mauldin-Williams:88} directly.
\end{proof}

\end{section}

\begin{section}{Relations between tent-tiles and Rauzy fractals}\label{sec:iwip}

In the present section we show the exact relation between tent-tiles and Rauzy fractals induced by four classes of substitutions. 


For an integer $p \geq 3$ let   $\zeta_p$ denote the substitution
\begin{equation}\label{zetap}
\zeta_p: 0 \mapsto 10, 1 \mapsto 2, \ldots, (p-2) \mapsto (p-1), (p-1) \mapsto (p-1)0.
\end{equation}
over the alphabet $\A=\{0, \ldots, p-1\}$.
The prefix graph $\Gamma_{\zeta_p}$ is sketched in Figure~\ref{Pgraph}.
\begin{figure}[ht]
\begin{tikzpicture}

\node[circle,draw=blue!60, very thick, minimum size=10mm] (v0) {\tiny $0$};
\node[circle,draw=blue!60, very thick, minimum size=10mm,right of = v0, node distance=2cm] (v1)  {\tiny $1$};
\node[right of = v1, node distance=2cm] (v2)  {};
\node[right of = v2, node distance=1cm] (vp-2)  {};
\node[circle,draw=blue!60, very thick, minimum size=10mm,right of = vp-2, node distance=2cm] (vp-1)  {\tiny $p-1$};

every edge quotes/.style = {font=\footnotesize, shorten >=1pt},

\draw[->,black, very thick] (v1) edge  node[pos=0.5, below=-2pt] {$\varepsilon$} (v0);
\draw[->,black, very thick, shorten <=0.3cm] (v2) edge  node[pos=0.65, below=-2pt] {$\varepsilon$} (v1);
\draw[-,black, very thick, dotted] (vp-2) edge (v2);
\draw[->,black, very thick, shorten >=0.3cm] (vp-1) edge  node[pos=0.35, below=-2pt] {$\varepsilon$} (vp-2);

\draw[->,black, very thick,bend left=20] (v0) edge  node[pos=0.2, above=-1pt] {$(p-1)$} (vp-1);

\draw[->,black, very thick, loop left] (v0) edge  node[pos=0.5, above=1pt] {$1$} (v0);

\draw[->,black, very thick, loop right] (vp-1) edge  node[pos=0.5, above=1pt] {$\varepsilon$} (vp-1);

\end{tikzpicture}

\caption{Prefix graph of the substitution $\zeta_p$. }
\label{Pgraph}
\end{figure}
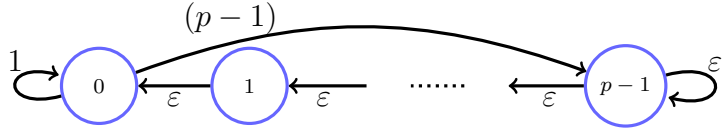

Note that the substitution $\zeta_p$ satisfies the strong coincidence condition. Indeed,  it is easy to see that
$\zeta_p^{p-1}(a)$ starts with the letter $p-1$ for each letter $a \in \A$.

\begin{theorem}\label{Rauzy}
Let $\alpha$ be a special Pisot unit such that $\alpha^p=\beta^2$ holds for an integer $p\geq 3$, i.e. $p \in \{3,4,6,8,10\}$.
Then $\zeta_p$ (as defined in \eqref{zetap}) is a primitive unimodular Pisot substitution with dominant root $\alpha$ and 
we have $\R_{\zeta_p}(k) \cong X_k$ with
\[
X_k=\begin{cases}  \fL^k \circ \fR(\Fa),  & \text{if  } k \in \{0, \ldots, p-2\}; \\
\fL^{p-1}(\Fa),& \text{if  } k =p-1.
\end{cases}\]
\end{theorem}

\begin{proof}
The substitution $\zeta_p$ is obviously primitive. The incidence matrix of $\zeta_p$ is given by
\[\bfM_{\zeta_p} = \left(\begin{array}{ccccc} 1 & 0 & \cdots & 0 & 1 \\1 & 0 & \cdots & 0 & 0 \\ 0 & 1 & \ddots & \vdots & \vdots \\
\vdots & \ddots & \ddots & 0 & 0 \\ 0 & \cdots & 0 & 1 & 1
\end{array}\right).\]
One easily verifies that the characteristic polynomial is given by $t^p-2t^{p-1}+t^{p-2}-1$ and that the dominant root of this polynomial is $\alpha$. Therefore, the substitution is a unimodular Pisot substitution and induces a Rauzy fractal. 

As $\alpha$ is the dominant root of $M_{\zeta_p}$ we see that $\KK_\alpha \cong \KK_{\zeta_p}$. The strategy of the proof is the following: for a specific linear bijection $\Phi: \KK_\alpha \longmapsto \KK_{\zeta_p}$ we show that the set list $\{\Phi(X_0),\ldots,\Phi(X_{p-1})\}$
satisfies the system of set equations \eqref{setequations} induced by the GIFS $(\Gamma_{\zeta_p},g)$. Then the assertion of the theorem follows directly from the uniqueness if the invariant set list in Proposition~\ref{GIFSpropo}.
Observe that all computations are actually number field calculations. 

Let $\bfnu:=(\beta ,\alpha, \alpha^2, \ldots, \alpha^{p-1})$. We claim that $\bfnu$ is a left eigenvector of $\bfM_{\zeta_p}$ with respect to $\alpha$.
Indeed, $\bfnu \, \bfM_{\zeta_p} = (\beta +\alpha, \alpha^2, \alpha^3, \ldots, \alpha^{p-1}, \beta +\alpha^{p-1})$ and by the definition of $\beta=\beta(\alpha)$ we have $\beta +\alpha = \alpha\beta$ and
$\beta+ \alpha^{p-1} = \alpha^p\beta^{-1} + \alpha^p\alpha^{-1} = \alpha^p$, hence $\bfnu \, \bfM_{\zeta_p} = \alpha \,\bfnu$.

We define $\Phi:\KK_\alpha \longmapsto \KK_{\zeta_p}$ to be the uniquely determined isomorphism that satisfies $\pi(x) = -\Phi\circ\bfPsi{x}$ for all $x \in \QQ(\alpha)$.
Then clearly $\Phi$ commutes $h$ and $\bfPhi{\alpha}$, that is
\[ h  \circ \Phi = \Phi\circ\bfPhi{\alpha}.\]

For each prefix $U \in \A^*$ we have
\[g_U: \KK_{\zeta_p} \longrightarrow \KK_{\zeta_p}, \bfx \longmapsto \pi\left(\spk{\bfnu, \bfell(U)}\right) + h(\bfx),\]
Therefore, for the appearing pairs of labels we easily obtain 
\[
\renewcommand{\arraystretch}{1.3}
\begin{array}{rll}
g_\varepsilon \circ \Phi(\bfx) &=  h\circ \Phi(\bfx) = \Phi\circ \bfPhi{\alpha}(\bfx) &= \Phi\circ \fL(\bfx),\\
g_1\circ \Phi(\bfx) & =  \pi\left(\spk{\bfnu, \bfell(1)}\right) + \Phi\circ\bfPhi{\alpha}(\bfx)  \\
& =  -\Phi \circ\bfPsi{\alpha} + \Phi\circ\bfPhi{\beta \beta^{-1}\alpha}(\bfx) \\
 & =\Phi\big(\bfPsi{\beta - \alpha\beta} +\bfPhi{\beta \alpha\beta^{-1}}(\bfx)\big) & = \Phi\circ \fR \circ \fL \circ \fR^{-1}  (\bfx),\\
g_{p-1}\circ \Phi(\bfx) & = \Phi\big(-\bfPsi{\alpha^{1-p}} + \bfPhi{\alpha}(\bfx)\big) \\
& = \Phi\big(-\bfPsi{\beta^2(1-\beta^{-1}}) + \bfPhi{\beta^2\alpha^{1-p}}(\bfx)\big) \\
&=  \Phi\big(\bfPsi{\beta - \beta^2} + \bfPhi{\beta^2\alpha^{1-p}}(\bfx)\big)
&= \Phi\circ \fR^2 \circ \fL^{1-p}(\bfx).
\end{array}\]

We are now in the position to verify that the system of set equations is satisfied.
Observe that $\fL(\Fa) \cup \fR(\Fa) = \Fa$.
For each $k \in \{1, \ldots, p-2\}$ we immediately see that $\Phi(X_k)=\Phi\circ \fL^{k} \circ \fR(\Fa) = 
g_\varepsilon \circ \Phi\circ\fL^{k-1} \circ \fR(\Fa) = g_\varepsilon \circ \Phi(X_{k-1})$.
Furthermore, we have
\begin{align*}
\Phi(X_0)=\Phi\circ \fR(\Fa) = & \Phi\circ \fR \circ \fL(\Fa) \cup  \Phi\circ \fR^2 (\Fa) \\
= &  \Phi\circ \fR \circ \fL \circ \fR^{-1} \circ \fR (\Fa) \cup  \Phi\circ \fR^2 \circ \fL^{1-p} \circ \fL^{p-1}(\Fa) \\ 
= & g_1\circ \Phi(X_0) \cup g_{p-1} \circ \Phi(X_{p-1}).
\end{align*}
Finally, $ \Phi(X_{p-1})=  \Phi\circ \fL^{p-1}(\Fa) =  \Phi\circ \fL^{p-1} \circ \fL (\Fa) \cup  \Phi\circ \fL^{p-1} \circ \fR(\Fa) = g_\varepsilon \circ\Phi (X_{p-1}) \cup g_\varepsilon\circ\Phi(X_{p-2})$.
\end{proof}

Observe that this result asserts that five of the tent-tiles are Rauzy fractals.

\begin{corollary}
The tent-tiles $\F_{\alpha_1}$, $\F_{\alpha_2}$, $\F_{\alpha_3}$, $\F_{\alpha_4}$ and $\F_{\alpha_5}$ are given - up to a linear isomorphism -  by the  Rauzy fractals induced by $\zeta_3$, $\zeta_4$, $\zeta_6$,
$\zeta_8$ and $\zeta_{10}$, respectively.
\end{corollary}
\begin{proof}
Table~\ref{dependencies} shows that the special Pisot units $\alpha_1$, $\alpha_2$, $\alpha_3$, $\alpha_4$ and $\alpha_5$ satisfy the conditions of Theorem~\ref{Rauzy} (for respective integers $p$). The identity $\Fa=X_0 \cup \cdots \cup X_{p-1}$ is  easy to see. 
\end{proof}

We introduce three further families of substitutions and show that they are intimately related with tent-tiles.
The proofs of the respective theorems run with the strategy used in the proof of Theorem~\ref{Rauzy}.
Observe that these families of substitutions are induced by automorphisms of the free group.
These automorphisms are stated explicitly in respective remarks at the end of the each  proof.

At first we consider the substitution $\theta_q$ over the alphabet 
$\A=\{0, \ldots, q-1, 2q-1\}$ defined by
\begin{equation}\label{thetaq}
\begin{array}{lrcl@{\qquad}rcl}
\theta_q: & 0 & \mapsto & 1, & q & \mapsto & (q+1),\\
 & 1 & \mapsto& 2 & (q+1), & \mapsto &  (q+2),\\
 & & \vdots & & & \vdots & \\
& (q-2) & \mapsto & (q-1), & (2q-2) & \mapsto & (2q-1), \\
& (q-1) & \mapsto & q(q-1), & (2q-1) & \mapsto & (2q-1)0. 
\end{array}
\end{equation}
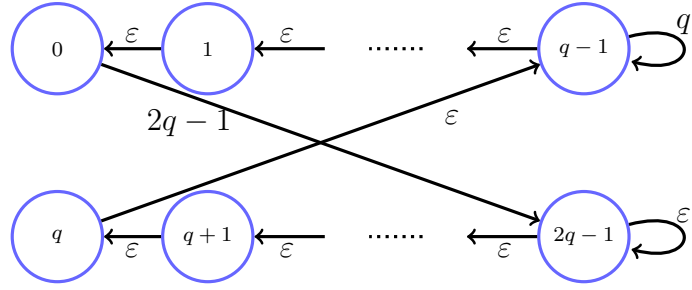
\begin{figure}[ht]
\begin{tikzpicture}

\node[circle,draw=blue!60, very thick, minimum size=12mm] (v0) {\tiny $0$};
\node[circle,draw=blue!60, very thick, minimum size=12mm,right of = v0, node distance=2cm] (v1)  {\tiny $1$};
\node[right of = v1, node distance=2cm] (v2)  {};
\node[right of = v2, node distance=1cm] (vq-2)  {};
\node[circle,draw=blue!60, very thick, minimum size=12mm,right of = vq-2, node distance=2cm] (vq-1)  {\tiny $q-1$};

\node[circle,draw=blue!60, very thick, minimum size=12mm, below of = v0, node distance=2.5cm] (vb0) {\tiny $q$};
\node[circle,draw=blue!60, very thick, minimum size=12mm,right of = vb0, node distance=2cm] (vb1)  {\tiny $q+1$};
\node[right of = vb1, node distance=2cm] (vb2)  {};
\node[right of = vb2, node distance=1cm] (vbq-2)  {};
\node[circle,draw=blue!60, very thick, minimum size=12mm,right of = vbq-2, node distance=2cm] (vbq-1)  {\tiny $2q-1$};

every edge quotes/.style = {font=\footnotesize, shorten >=1pt},

\draw[->,black, very thick] (v1) edge  node[pos=0.5, above=-2pt] {$\varepsilon$} (v0);
\draw[->,black, very thick, shorten <=0.3cm] (v2) edge  node[pos=0.65, above=-2pt] {$\varepsilon$} (v1);
\draw[-,black, very thick, dotted] (vq-2) edge (v2);
\draw[->,black, very thick, shorten >=0.3cm] (vq-1) edge  node[pos=0.35, above=-2pt] {$\varepsilon$} (vq-2);

\draw[->,black, very thick, loop right] (vq-1) edge  node[pos=0.5, above=1pt] {$q$} (vq-1);

\draw[->,black, very thick] (vb1) edge  node[pos=0.5, below=-2pt] {$\varepsilon$} (vb0);
\draw[->,black, very thick, shorten <=0.3cm] (vb2) edge  node[pos=0.65, below=-2pt] {$\varepsilon$} (vb1);
\draw[-,black, very thick, dotted] (vbq-2) edge (vb2);
\draw[->,black, very thick, shorten >=0.3cm] (vbq-1) edge  node[pos=0.35, below=-2pt] {$\varepsilon$} (vbq-2);

\draw[->,black, very thick, loop right] (vbq-1) edge  node[pos=0.5, above=1pt] {$\varepsilon$} (vbq-1);

\draw[->,black, very thick] (v0) edge  node[pos=0.2, below=1pt] {$2q-1$} (vbq-1);

\draw[->,black, very thick] (vb0) edge  node[pos=0.8, below=1pt] {$\varepsilon$} (vq-1);

\end{tikzpicture}

\caption{Prefix graph of the substitution $\theta_q$. }
\label{thetaqgraph}
\end{figure}
Its prefix graph $\Gamma_{\theta_q}$ is depicted in Figure~\ref{thetaqgraph}. Note that $\theta_q^{2q}(k)$ starts with the letter $2q-1$ for each $k \in \A$. Therefore, $\zeta_q$ satisfies the strong coincidence condition.

\begin{theorem}\label{iwip1}
Let $\alpha$ be a special Pisot unit such that $\alpha^q=\beta$ holds for an integer $q \geq 2$, i.e. $q \in \{2,3,4,5\}$,
and define for each $k \in \A:=\{0, \ldots, 2q-1\}$ the set $X_k$ by
\[
X_k = \begin{cases}
 \bfPsi{\alpha^k}-\fL^k \circ \fR(\Fa), & \text{if } k \in \{0, \ldots, q-2\}; \\
\bfPsi{\alpha^{q-1}}-\fL^{q-1} (\Fa), & \text{if } k =q-1;\\
\fL^{k-q} \circ \fR(\Fa), & \text{if } k \in \{q, \ldots, 2q-2\}; \\
\fL^{q-1} (\Fa), & \text{if } k =2q-1. 
\end{cases}\] 
Then $\theta_q$ (as defined in \eqref{thetaq}) is a primitive unimodular Pisot substitution with dominant root $\alpha$
such that $\R_{\theta_q}(k) \cong X_k$ holds (up to isomorphism) for each $k \in \A$.
\end{theorem}

\begin{proof}
The substitution $\theta_q$ is obviously primitive. Its incidence matrix  is given by the block matrix
\[\bfM_{\theta_q}=\left(\begin{array}{cc} \bfM^+_{q}  & \bfM^-_{q}\\ \bfM^-_{q}  & \bfM^+_{q} \end{array}\right),\]
where
\begin{equation}\label{M+M-}
\bfM^+_{{q}} = \left(\begin{array}{ccccc}  0 & \cdots &\cdots & 0 & 0 \\ 1 & \ddots & & \vdots & \vdots \\ 0 & \ddots & \ddots & \vdots & \vdots\\
\vdots & \ddots & \ddots & 0 & 0 \\ 0 & \cdots & 0 & 1 & 1
\end{array}\right) \in \RR^{q\times q} \text{ and }
\bfM^-_{{q}} = \left(\begin{array}{cccc}  0  &\cdots & 0 & 1 \\ \vdots & &  \vdots & 0 \\ \vdots & &  \vdots & \vdots\\
0 & \cdots & 0 & 0
\end{array}\right)\in \RR^{q\times q}.
\end{equation}

By observing the proof of \cite[Proposition~4.9]{Arnoux-Berthe-Hilion-Siegel.2006} we see that the dominant eigenvalue of $\bfM_{\theta_q}$ is given by the dominant eigenvalue of $\bfM_{q}^+ + \bfM^-_{q}$ which can easily seen to be the dominant root
of $t^q-t^{q-1} -1$. From this we conclude that this root must coincide with $\alpha$ and we have $\KK_\alpha \cong \KK_{\theta_q}$.
We proceed as in  Theorem~\ref{Rauzy} and show that the set list $\{\Psi(X_k): k \in \A\}$ satisfies the system set
equations \eqref{setequations} induces by the GIFS $(\Gamma_{\theta_q}, g)$ for a suitable isomorphism $\Phi:\KK_\alpha \longmapsto \KK_{\theta_q}$. In particular, as before we define $\Phi$ to be the uniquely determined isomorphism that satisfies $\pi(x)=-\Phi\circ\bfPsi{x}$ for all $x \in \QQ(\alpha)$.
Again we have $h \circ \Phi = \Phi \circ\bfPhi{\alpha}$.

Let $\bfnu=(1, \alpha, \ldots, \alpha^{q-1})$ and observe that $(\bfnu, \bfnu)$ is a left eigenvector of $\bfM_{\theta_q}$ with respect to $\alpha$. 
The prefixes that occur in the prefix graph $\Gamma_{\theta_q}$ (see Figure~\ref{thetaqgraph})  are $\varepsilon, q, 2q-1$. By considering the realisation $g$ we obtain for each $\bfx \in \RR^d$
\[
\renewcommand{\arraystretch}{1.3}
\begin{array}{rll}
g_\varepsilon \circ \Phi(\bfx)& = \Phi\circ\bfPhi{\alpha}(\bfx)   & = \Phi\circ\fL(\bfx), \\
g_{q}\circ \Phi(\bfx)& = \Phi\big(-\bfPsi{1} + \bfPhi{\alpha}(\bfx)\big) &= \Phi\big(-\bfPsi{1} + \fL(\bfx)\big), \\ 
g_{2q-1}\circ \Phi(\bfx)&=   
\Phi\big(-\bfPsi{\alpha^{q-1}} +  \bfPsi{\alpha}(\bfx)\big) \\
&=  \Phi\big(\bfPsi{1-\alpha^{q}} + \bfPsi{\beta\alpha^{1-q}}(\bfx) \big)
&= \Phi\big(\bfPsi{1} - f_R\circ f_L^{1-q}(\bfx)\big).
\end{array}\]

Now, for $k=0$ we easily calculate
\begin{align*}
\Phi(X_0) = &  \Phi\big(\bfPsi{1}  - \fR(\Fa)\big) = \Phi\big(\bfPsi{1}-\fR\circ \fL^{1-q}\circ\fL^{q-1}(\Fa)\big)
=  g_{2q-1}\circ\Phi(X_{2q-1}).
\end{align*}
For  $k \in \{1, \ldots, q-2\}$ we have
\begin{align*}
\Phi(X_{k}) & = \Phi\big(\bfPsi{\alpha^{k}} -\fL^{k}\circ \fR(\Fa)\big) =
\Phi\big(\bfPhi{\alpha}\circ\bfPsi{\alpha^{k-1}} - \bfPhi{\alpha}\circ \fL^{k-1}\circ \fR(\Fa)\big) \\
&= \Phi\circ\fL\left(\bfPsi{\alpha^{k-1}} - \fL^{k-1}\circ \fR(\Fa)\right)
= g_\varepsilon\circ\Phi(X_{k-1}).
\end{align*}
For  $k=q-1$ we obtain
\begin{align*}
\Phi(X_{q-1}) = & \Phi\big(\bfPsi{\alpha^{q-1}} - \fL^{q-1}(\Fa)\big) \\
 = & \Phi\big(\bfPsi{\alpha^{q-1}} - \fL^{q-1}\circ \fR(\Fa)\big) \cup \Phi\big(\bfPsi{\alpha^{q-1}} -\fL^{q}(\Fa)\big) \\
= & \Phi\left(\bfPsi{\alpha^{q-1}} - \fL^{q-1}\circ \fR(\Fa)\right) \cup \Phi\big( \bfPsi{1 - \alpha^q}  +\fL^{q}(\Fa)\big)\\
= & \Phi\circ \fL \big(\bfPsi{\alpha^{q-2}} - \fL^{q-2}\circ \fR(\Fa)\big) \cup \Phi\left( -\bfPsi{1} +\fL\left(\bfPsi{\alpha^{q-1}}  -\fL^{q-1}(\Fa) \right)\right) \\
= & g_{\varepsilon}\circ\Phi(X_{q-2})  \cup g_{q}\circ\Phi(X_{q-1}).
\end{align*}
For $k=q$ we see that
\begin{align*}
\Phi(X_{q}) = & \Phi\circ\fR(\Fa) = \Phi\big(\fL^q \,\left(\bfPsi{1} - \Fa\right)\big) \\
= & \Phi\circ\fL\left(\bfPsi{\alpha^{q-1}} - \fL^{q-1}(\Fa)\right) =  g_{\varepsilon}\circ\Phi(X_{q-1}) .
\end{align*}
Now we let $k \in \{q+1, \ldots, 2q-1\}$. Here
trivial computations yield
\[
\Phi(X_{k}) =   \Phi\circ\fL^{k-q}\circ \fR(\Fa) =  \Phi\circ\fL \circ \fL^{k-q-1}\circ \fR(\Fa)
=  g_\varepsilon\circ\Phi(X_{k-1}).
\]
Finally, for  $k=2q-1$ we have
\[
\Phi(X_{2q-1}) =   \Phi\circ\fL^{q-1}(\Fa) = \Phi\left(\fL^{q-1}\circ \fR(\Fa)\right) \cup  \Phi\circ\fL^{q}(\Fa)\\
 = 
g_{\varepsilon}\circ\Phi(X_{2q-2}) \cup 
g_{\varepsilon}\circ\Phi(X_{2q-1}).
\]
\end{proof}
The following observation can be obtained immediately from the theorem. 
\begin{corollary}\label{iwip1cor}
Let $\alpha$ be a special Pisot unit such that $\alpha^q=\beta$ holds for an integer $q\geq 2$.
Then $\Fa$ coincides with $\R_{\theta_q}(q) \cup \cdots \cup  \R_{\theta_q}(2q-1)$ up to a linear transformation.
\end{corollary}

\begin{remark}\label{remark1}
Observe that the substitution $\theta_q$ is induced by an automorphism of the free group. In particular, $\theta_q$ is the (non-orientable) double substitution (in the sense of
\cite[Definition~3.1]{Arnoux-Berthe-Hilion-Siegel.2006}) of the automorphism $\varphi_q$ of the free group generated by $\{0, \ldots, q-1\}$ defined by
\[\varphi_q: 0 \mapsto 1, 1 \mapsto 2, \ldots, (q-2) \mapsto (q-1), (q-1) \mapsto (q-1)1^{-1}.\]
The case  $\varphi_5$ is discussed in \cite[Example~5.4]{Arnoux-Berthe-Hilion-Siegel.2006} and we find our tent-tile $\F_{\alpha_5}$ in the respective figure.
\end{remark}

Lets turn to the next family of substitutions that we denote by $\theta'_q$ with $q$ a positive integer. The exact definition depends on whether $q$ is odd or even. 

\begin{subequations}
If $q \equiv 1 \, (\mod 2)$ then $\theta'_q$ is defined over the alphabet $\A=\{0, \ldots, 2q-1\}$ by
\begin{equation}\label{thetasta}
\begin{array}{cl@{\,\mapsto\,}lcl@{\,\mapsto\,}ll@{\,\mapsto\,}l}
\theta'_q: & 0  &(q+1),  & \ldots, & (q-2) & (2q-1), & (q-1) & 0(2q-1) \\
& q & 1,  & \ldots, & (2q-2)& (q-1), & (2q-1)& (q-1)q.
\end{array} 
\end{equation}
For $q \equiv 0 \, (\mod 2)$ the substitution  $\theta'_q$ is defined over the alphabet $\A=\{0, \ldots, q-1\}$ by
\begin{equation}\label{thetastb}
\theta'_q: 0 \mapsto 2, 1 \mapsto 3, \ldots, (q-3) \mapsto (q-1), (q-2) \mapsto 0(q-1), (q-1) \mapsto 0(q-1)1.
\end{equation}
\end{subequations}
For the prefix graph see Figure~\ref{thetaqsgraph}.
Note that $\theta'_2: 0 \mapsto 01, 1 \mapsto 011$.


\begin{figure}[ht]
\vskip 0cm
\begin{tikzpicture}

\node[circle,draw=blue!60, very thick, minimum size=12mm] (v0) {\tiny $0$};
\node[circle,draw=blue!60, very thick, minimum size=12mm,right of = v0, node distance=1.5cm] (v1)  {\tiny $1$};
\node[right of = v1, node distance=1.5cm] (v2)  {};
\node[right of = v2, node distance=0.7cm] (vq-2)  {};
\node[circle,draw=blue!60, very thick, minimum size=12mm,right of = vq-2, node distance=1.5cm] (vq-1)  {\tiny $q-1$};

\node[circle,draw=blue!60, very thick, minimum size=12mm, below of = v0, node distance=2cm] (vb0) {\tiny $q$};
\node[circle,draw=blue!60, very thick, minimum size=12mm,right of = vb0, node distance=1.5cm] (vb1)  {\tiny $q+1$};
\node[right of = vb1, node distance=1.5cm] (vb2)  {};
\node[right of = vb2, node distance=0.7cm] (vbq-2)  {};
\node[circle,draw=blue!60, very thick, minimum size=12mm,right of = vbq-2, node distance=1.5cm] (vbq-1)  {\tiny $2q-1$};

\node[circle,draw=blue!60, very thick, minimum size=10mm, right of = vq-1, node distance=3cm] (vx0) {\tiny $0$};
\node[circle,draw=blue!60, very thick, minimum size=10mm,right of = vx0, node distance=1.7cm] (vx2)  {\tiny $2$};
\node[right of = vx2, node distance=1.7cm] (vx4)  {};
\node[right of = vx4, node distance=0.7cm] (vxq-4)  {};
\node[circle,draw=blue!60, very thick, minimum size=10mm,right of = vxq-4, node distance=1.7cm] (vxq-2)  {\tiny $q-2$};

\node[circle,draw=blue!60, very thick, minimum size=10mm, below of = vx0, node distance=2cm] (vx1) {\tiny $1$};
\node[circle,draw=blue!60, very thick, minimum size=10mm,right of = vx1, node distance=1.7cm] (vx3)  {\tiny $3$};
\node[right of = vx3, node distance=1.7cm] (vx5)  {};
\node[right of = vx5, node distance=0.7cm] (vxq-3)  {};
\node[circle,draw=blue!60, very thick, minimum size=10mm,right of = vxq-3, node distance=1.7cm] (vxq-1)  {\tiny $q-1$};

every edge quotes/.style = {font=\footnotesize, shorten >=1pt},

\draw[->,black, very thick] (v1) edge  node[pos=0.1, below=0pt] {$\varepsilon$} (vb0);
\draw[->,black, very thick, shorten <=0.45cm] (v2) edge  node[pos=0.3, below=0pt] {$\varepsilon$} (vb1);
\draw[-,black, very thick, dotted] (vq-2) edge (v2);
\draw[->,black, very thick, shorten >=0.45cm] (vq-1) edge  node[pos=0.7, above=0pt] {$\varepsilon$} (vbq-2);

\draw[->,black, very thick] (vb1) edge  node[pos=0.1, above=0pt] {$\varepsilon$} (v0);
\draw[->,black, very thick, shorten <=0.45cm] (vb2) edge  node[pos=0.3, above=0pt] {$\varepsilon$} (v1);
\draw[-,black, very thick, dotted] (vbq-2) edge (vb2);
\draw[->,black, very thick, shorten >=0.45cm] (vbq-1) edge  node[pos=0.7, below=0pt] {$\varepsilon$} (vq-2);

\draw[->,black, very thick, bend right=10] (vq-1) edge  node[pos=0.5, left=-3pt] {$\varepsilon$} (vbq-1);
\draw[->,black, very thick, bend right=10] (vbq-1) edge  node[pos=0.5, right=-3pt] {$0$} (vq-1);
\draw[->,black, very thick, bend right=30] (vb0) edge  node[pos=0.2, below=0pt] {$q-1$} (vbq-1);
\draw[->,black, very thick, bend left=30] (v0) edge  node[pos=0.2, above=0pt] {$\varepsilon$} (vq-1);

\draw[->,black, very thick] (vx2) edge  node[pos=0.5, above=-2pt] {$\varepsilon$} (vx0);
\draw[->,black, very thick, shorten <=0.3cm] (vx4) edge  node[pos=0.65, above=-2pt] {$\varepsilon$} (vx2);
\draw[-,black, very thick, dotted] (vxq-4) edge (vx4);
\draw[->,black, very thick, shorten >=0.3cm] (vxq-2) edge  node[pos=0.35, above=-2pt] {$\varepsilon$} (vxq-4);

\draw[->,black, very thick] (vx3) edge  node[pos=0.5, below=-2pt] {$\varepsilon$} (vx1);
\draw[->,black, very thick, shorten <=0.3cm] (vx5) edge  node[pos=0.65, below=-2pt] {$\varepsilon$} (vx3);
\draw[-,black, very thick, dotted] (vxq-3) edge (vx5);
\draw[->,black, very thick, shorten >=0.3cm] (vxq-1) edge  node[pos=0.35, below=-2pt] {$\varepsilon$} (vxq-3);

\draw[->,black, very thick, loop right] (vxq-1) edge  node[pos=0.5, above=1pt] {$0$} (vxq-1);
\draw[->,black, very thick, bend left=25] (vx0) edge  node[pos=0.2, above=0pt] {$\varepsilon$} (vxq-2);
\draw[->,black, very thick, bend right=25] (vx1) edge  node[pos=0.2, below=0pt] {$0(q-1)$} (vxq-1);
\draw[->,black, very thick] (vx0) edge  node[pos=0.5, above=-2pt] {$\varepsilon$} (vxq-1);
\draw[->,black, very thick] (vxq-1) edge  node[pos=0.5, right=-3pt] {$0$} (vxq-2);
\end{tikzpicture}
\caption{Prefix graph of the substitution $\theta'_q$ for $q$ odd (left) and $q\geq 4$ even (right). }
\label{thetaqsgraph}
\end{figure}
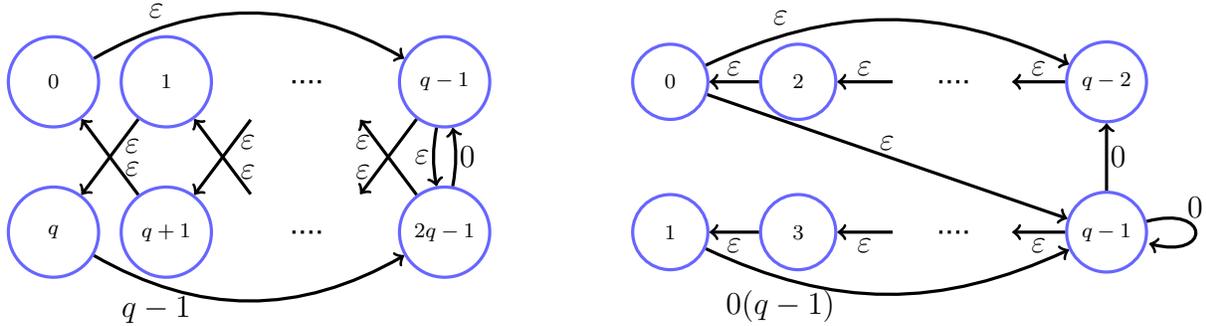

\begin{theorem}\label{iwip2}
Let $\alpha$ be a special Pisot unit such that $\alpha=\beta^q$ holds for an integer $q\geq 2$, i.e. $q \in \{2,3,4,5\}$.
Then $\theta'_q$ (as defined in \eqref{thetasta} and \eqref{thetastb}, respectively) is a primitive unimodular Pisot substitution. The dominant root is $\beta$ if $q$ is odd, and $\beta^2$ if $q$ is even.
Furthermore, $\R_{\theta_q}(k) \cong  X_k$ holds for each $k \in \A$ with
\[
X_k = \begin{cases}
\bfPhi{\alpha\beta^k}( \Fa),  & \text{if } k \in \{0, 2, \ldots, q-3\}; \\
\bfPsi{\beta^{k}}-\bfPhi{\alpha\beta^k}( \Fa),   & \text{if } k \in \{1, 3, \ldots, q-2\}; \\
\bfPhi{\beta^{q-1}}( \Fa),  & \text{if } k=q-1; \\
\bfPsi{\beta^{k-q}}-\bfPhi{\alpha\beta^{k-q}}( \Fa),  & \text{if } k \in \{q, q+2, \ldots, 2q-3\}; \\
\bfPhi{\alpha\beta^{k-q}}(\Fa),  & \text{if } k \in \{q+1, q+3, \ldots, 2q-2\}; \\
\bfPsi{\beta^{q-1}}-\bfPhi{\beta^{q-1}} (\Fa),   & \text{if } k =2q-1;
\end{cases}\]
for odd $q$ and
\[
X_k = \begin{cases}
\bfPhi{\alpha\beta^k} ( \Fa), & \text{if } k \in \{0, 2, \ldots, q-2\}; \\
\bfPsi{\beta^{k}}-\bfPhi{\alpha\beta^k} ( \Fa), & \text{if } k \in \{1, 3, \ldots, q-3\}; \\
\bfPsi{\beta^{q-1}}-\bfPhi{\beta^{q-1}} ( \Fa), & \text{if } k=q-1;
\end{cases}\]
for even $q$.
\end{theorem}
\begin{proof}
We proceed as in the Theorems~\ref{Rauzy} and \ref{iwip1} and show that for a 
suitable isomorphism $\Phi$
the set list $\{\Phi(X_k): k \in \A\}$ satisfies the system of set equations \eqref{setequations} induced by the GIFS $(\Gamma_{\theta'_q}, g)$.

We start with the  case $q \equiv 1 (\mod 2)$. The primitivity is easy to see.
We also see that for the incidence matrix we have
\[\bfM_{\theta'_q}=\left(\begin{array}{cc} \bfM^-_{q} & \bfM^+_{q}  \\  \bfM^+_{q} & \bfM^-_{q}\end{array}\right)\]
with $\bfM^+_{q}$ and $\bfM^-_{q}$ as defined in \eqref{M+M-}.
This immediately shows that $\beta$ is the dominant eigenvalue of $\bfM_{\theta'_q}$, hence, $\theta'_q$ is a Pisot substitution. Since $\QQ(\beta) =\QQ(\alpha)$ we obviously have that $\KK_\alpha \cong \KK_{\theta'_q}$.

For $\bfnu:=(1, \beta, \ldots, \beta^{q-1})$ the vector $(\bfnu, \bfnu)$ is a left eigenvector of $\bfM_{\theta'_q}$ with respect to $\beta$. 
Similar as before we define $\Phi:\KK_\alpha \longmapsto \KK_{\theta'_p}$ to be the uniquely determined isomorphism that satisfies $\pi(x)=-\Phi\circ\bfPsi{x}$ for all $x \in \QQ(\alpha)$. Observe that here $\Phi$ commutes $h$ and $\bfPhi{\beta}$, i.e. $h \circ \Phi = \Phi \circ\bfPhi{\beta}$.

The prefixes that occur in 
the prefix graph $\Gamma_{\theta'_q}$ (see Figure~\ref{thetaqsgraph})  are $\varepsilon, 0, q-1$. For the realisation $g$ we thus obtain
\begin{align*}
g_\varepsilon\circ \Phi(\bfx) = & \Phi\circ\bfPhi{\beta}(\bfx), \\
g_{0}\circ \Phi(\bfx)  = & \Phi\left(-\bfPsi{1} + \bfPhi{\beta}(\bfx) \right), \\ 
g_{q-1}\circ \Phi(\bfx) = &  \Phi\left(-\bfPsi{\beta^{q-1}} + \bfPhi{\beta}(\bfx)\right).
\end{align*}

For each $k \in \A=\{0,\ldots, 2q-1\}$ we have a set equation to verify.
For $k=0$ we immediately see that
\[\Phi(X_0)=\Phi\circ\bfPhi{\alpha}(\Fa) = \Phi\circ\bfPhi{\beta^q}(\Fa) = \Phi\circ\bfPhi{\beta}\circ\bfPhi{\beta^{q-1}}(\Fa) =  g_\varepsilon\circ\Phi(X_{q-1}).\]
The cases $k \in \{1, \ldots q-2\}$ can be shown without any problems.
For $k=q-1$ we have
\begin{align*}
\Phi(X_{q-1})  & = \Phi\circ\bfPhi{\beta^{q-1}} \big(\fL(\Fa) \cup \fR(\Fa)\big) = g_{\varepsilon}\circ\Phi\circ\bfPhi{\beta^{q-2}} \big(\fL(\Fa) \cup \fR(\Fa)\big) \\
 &=  g_{\varepsilon}\circ\Phi\circ\bfPhi{\alpha\beta^{q-2}}(\Fa) \cup g_{\varepsilon}\circ\Phi\big(\bfPsi{\beta^{q-1}}-\bfPhi{\beta^{q-1}}(\Fa)\big) \\
 & = g_{\varepsilon}\circ\Phi(X_{2q-2}) \cup  g_{\varepsilon}\circ\Phi(X_{2q-1}).
\end{align*}
For the case $k=q$ we observe that $\beta^q-\beta^{q-1}=1$ and obtain

\begin{align*}
\Phi(X_q)& =\Phi\big(\bfPsi{1} - \bfPhi{\alpha}(\Fa)\big) = \Phi\big(- \bfPsi{\beta^{q-1} + \beta^{q}} - \bfPhi{\beta^q}(\Fa)\big) \\
& = g_{q-1} \circ\Phi\big(\bfPsi{\beta^{q-1}} - \bfPhi{\beta^{q-1}}(\Fa)\big)
= g_{q-1} \circ\Phi(X_{2q-1}).
\end{align*}
For $k \in \{q+1, \ldots 2q-2\}$ the proof is straightforward.
Finally, for $k=2q-1$ we see that 
\begin{align*}
\Phi(X_{2q-1})& =\Phi\big(\bfPsi{\beta^{q-1}} -\bfPhi{\beta^{q-1}}(\Fa) \big) = \Phi\big(\bfPsi{\beta^{q-1}} -\bfPhi{\beta^{q-1}}\left(\fL(\Fa) \cup \fR(\Fa)\right)\big)  \\
&=  \Phi\circ \bfPhi{\beta}\left(\bfPsi{\beta^{q-2}} -\bfPhi{\alpha\beta^{q-2}} (\Fa) \right) \cup  \Phi\left(\bfPsi{\beta^{q-1} - \beta^q} +\bfPhi{\beta^{q}}(\Fa)\right) \\
&=  g_{\varepsilon}\circ\Phi\big(\bfPsi{\beta^{q-2}} -\bfPhi{\alpha\beta^{q-2}} (\Fa) \big) \cup  \Phi\big(-\bfPsi{1} +\bfPhi{\beta\beta^{q-1}}(\Fa)\big) \\
 &=  g_{\varepsilon}\circ\Phi(X_{q-2}) \cup  g_{0}\circ\Phi(X_{q-1}).
\end{align*}

For even $q$ the substitution $\theta'_q$ has a similar shape as $(\theta'_q)^2$ for odd $q$. 
Indeed, the incidence matrix of $\theta'_q$ is given by
\[\bfM_{\theta'_q} = (\bfM^+_{q} + \bfM^-_{q})^2\]
and, therefore, the dominant root is $\beta^2$.
We have $\QQ(\beta^2) =\QQ(\alpha)$, thus,  $\KK_\alpha \cong \KK_{\theta'_q}$.

We perform the proof for $q=4$ here. The case $q=2$ is degenerated and, hence, differs in some details. We leave the explicit proof for $q=2$ to the interested reader.

The vector $\bfnu$ from the odd case is a left eigenvector with respect to $\beta^2$.
We define $\Phi:\KK_\alpha \longmapsto \KK_{\theta'_4}$ such that $\pi(x)=-\Phi\circ\bfPsi{x}$ for all $x \in \QQ(\alpha)$. Note that we have $h \circ \Phi = \Phi \circ\bfPhi{\beta^2}$, i.e.
\[g_U \circ \Phi(\bfx)= \pi(\spk{\bfnu, \bfell(U)}) + h\circ\Phi(\bfx) =   \Phi\left(-\bfPsi{\spk{\bfnu,\bfell(U)}}+\bfPhi{\beta^2}(\bfx)\right),\]
where $U \in \{\varepsilon, 0, 03\}$ is one of the occurring prefixes. 

For $k=0$ we calculate
\begin{align*}
\Phi(X_0) = & \Phi\circ\bfPhi{\alpha}(\Fa) = \Phi\circ\bfPhi{\beta^q}(\Fa)=
g_{\varepsilon}\circ \Phi\circ\bfPhi{\beta^{2}}\big(\fL(\Fa) \cup \fR(\Fa)\big) \\
= & g_{\varepsilon}\circ \Phi\circ\bfPhi{\alpha\beta^{2}}(\Fa) \cup
g_{\varepsilon}\circ \Phi\big(\bfPsi_\alpha(\beta^{3}) - \bfPhi{\beta^{3}}(\Fa)\big)\\
= & g_{\varepsilon}\circ \Phi(X_{2}) \cup g_{\varepsilon} \circ \Phi(X_{3}).
\end{align*}
For showing the equation associated with $k = 1$ we again observe that $\beta^q=\beta^{q-1}+1$.
\begin{align*}
\Phi(X_1)  & =\Phi\big(\bfPsi{\beta}-\bfPhi{\alpha\beta}(\Fa)\big)
=  \Phi\big(\bfPsi{-\beta^4+\beta^{5}}-\bfPhi{\beta^{5}}(\Fa)\big) \\
 &= \Phi\big(\bfPsi{-1-\beta^{3}}+\bfPhi{\beta^{2}}(\bfPsi{\beta^{3}}-\bfPhi{\beta^{3}}(\Fa))\big)
= g_{03} \circ \Phi(X_{3}).
\end{align*}
For $k =2$ the verification of the set equation is trivial. Thus, lets turn to the case $k=3$. Here we have
\begin{align*}
\Phi(X_{3}) =& 
\Phi\big(\bfPsi{\beta^{3}} - \bfPhi{\beta^{3}}\left(\fL(\Fa) \cup \fR(\Fa)\right) \big)\\
= & g_\varepsilon\circ\Phi\big(\bfPsi{\beta^{1}} - \bfPhi{\alpha\beta^{1}}(\Fa)\big) \cup \Phi\big(-\bfPsi{1} + \bfPhi{\beta^{4}}(\Fa)\big) \\
= & g_{\varepsilon}\circ \Phi(X_{1}) \cup g_{0}\circ \Phi\circ\bfPhi{\beta^{2}}\left(\fL(\Fa) \cup \fR(\Fa)\right) = \\
= & g_{\varepsilon}\circ \Phi(X_{1}) \cup g_{0}\circ \Phi\circ\bfPhi{\alpha\beta^{2}}(\Fa)  \cup g_{0}\circ \Phi\big(\bfPsi{\beta^{3}} - \bfPhi{\beta^{3}}(\Fa)\big)
\\
= & g_{\varepsilon}\circ \Phi(X_{1}) \cup 
g_{0}\circ \Phi(X_{2}) \cup g_{0}\circ \Phi(X_{3}).
\end{align*}
\end{proof}

\begin{remark}\label{remark2}
The substitution $\theta'_q$ is the double substitution of the automorphism  of the free group generated by $\{0, \ldots, q-1\}$ 
\[\varphi'_q : \begin{cases}
0 \mapsto 1^{-1}, 1 \mapsto 2^{-1},\ldots, (q-2) \mapsto (q-1)^{-1}, (q-1) \mapsto 0(q-1)^{-1}, & \text{if $q$ is odd}; \\
0 \mapsto 1^{-1}, 1 \mapsto 2^{-1},\ldots, (q-2) \mapsto (q-1)^{-1}, (q-1) \mapsto (q-1)^{-1}0^{-1}, & \text{if $q$ is even}.
\end{cases}
\]
If $q$ is even then this double-substitution is orientable (more precisely, orientation reversing, see 
\cite[Definition~3.5]{Arnoux-Berthe-Hilion-Siegel.2006}) and, hence, not primitive. In fact, one easily verifies, that in this case we have
$\theta'_q= \varphi'_q \circ \varphi'_q$.
\end{remark}

\begin{subequations}
Finally, we consider the substitution $\zeta'_p$, where $p \geq 2$ is an integer. For
odd $p$ it is defined over the alphabet $\A=\{0, \ldots, 2p-1\}$ by
\begin{equation}\label{setasta}
\begin{array}{rrcl@{\hspace{1cm}}rcl}
\zeta'_p: & 0 & \mapsto &  (p+1)\, p, & p & \mapsto & 01, \\
&1 & \mapsto & (p+2), & (p+1) & \mapsto & 2, \\
&& \vdots & & &\vdots & \\
&(p-2) &\mapsto & (2p-1), & (2p-2) &\mapsto& (p-1), \\
&(p-1) & \mapsto & 0 \, (2p-1), & (2p-1) & \mapsto &(p-1) \, p.
\end{array}
\end{equation}
If $p$ is even then $\zeta'_p$ is defined over the alphabet $\A=\{0, \ldots, p-1\}$ by
\begin{equation}\label{setastb}
\begin{array}{rrcl}
\zeta'_p: & 0 & \mapsto & 201, \\
& 1 & \mapsto  & 3, \\
&& \vdots & \\
 & (p-3) &  \mapsto & (p-1),\\
& (p-2) & \mapsto & 0(p-1),\\ 
& (p-1) & \mapsto & 0(p-1)01.
\end{array}
\end{equation}
\end{subequations}

\begin{figure}[ht]

\begin{tikzpicture}

\node[circle,draw=blue!60, very thick, minimum size=12mm] (v0) {\tiny $0$};
\node[circle,draw=blue!60, very thick, minimum size=12mm,right of = v0, node distance=1.5cm] (v1)  {\tiny $1$};
\node[right of = v1, node distance=1.5cm] (v2)  {};
\node[right of = v2, node distance=0.7cm] (vp-2)  {};
\node[circle,draw=blue!60, very thick, minimum size=12mm,right of = vp-2, node distance=1.5cm] (vp-1)  {\tiny $p-1$};

\node[circle,draw=blue!60, very thick, minimum size=12mm, below of = v0, node distance=2cm] (vb0) {\tiny $p$};
\node[circle,draw=blue!60, very thick, minimum size=12mm,right of = vb0, node distance=1.5cm] (vb1)  {\tiny $p+1$};
\node[right of = vb1, node distance=1.5cm] (vb2)  {};
\node[right of = vb2, node distance=0.7cm] (vbp-2)  {};
\node[circle,draw=blue!60, very thick, minimum size=12mm,right of = vbp-2, node distance=1.5cm] (vbp-1)  {\tiny $2p-1$};

\node[circle,draw=blue!60, very thick, minimum size=10mm, right of = vq-1, node distance=2.5cm] (vx0) {\tiny $0$};
\node[circle,draw=blue!60, very thick, minimum size=10mm,right of = vx0, node distance=1.7cm] (vx2)  {\tiny $2$};
\node[right of = vx2, node distance=1.7cm] (vx4)  {};
\node[right of = vx4, node distance=0.7cm] (vxp-4)  {};
\node[circle,draw=blue!60, very thick, minimum size=10mm,right of = vxp-4, node distance=1.7cm] (vxp-2)  {\tiny $p-2$};

\node[circle,draw=blue!60, very thick, minimum size=10mm, below of = vx0, node distance=2cm] (vx1) {\tiny $1$};
\node[circle,draw=blue!60, very thick, minimum size=10mm,right of = vx1, node distance=1.7cm] (vx3)  {\tiny $3$};
\node[right of = vx3, node distance=1.7cm] (vx5)  {};
\node[right of = vx5, node distance=0.7cm] (vxp-3)  {};
\node[circle,draw=blue!60, very thick, minimum size=7mm,right of = vxp-3, node distance=1.7cm] (vxp-1)  {\tiny $p-1$};

every edge quotes/.style = {font=\footnotesize, shorten >=1pt},

\draw[->,black, very thick] (v1) edge  node[pos=0.05, below=0pt] {$0$} (vb0);
\draw[->,black, very thick, shorten <=0.45cm] (v2) edge  node[pos=0.3, below=0pt] {$\varepsilon$} (vb1);
\draw[-,black, very thick, dotted] (vp-2) edge (v2);
\draw[->,black, very thick, shorten >=0.45cm] (vp-1) edge  node[pos=0.7, above=0pt] {$\varepsilon$} (vbp-2);

\draw[->,black, very thick] (vb1) edge  node[pos=0.05, above=0pt] {$\varepsilon$} (v0);
\draw[->,black, very thick, shorten <=0.45cm] (vb2) edge  node[pos=0.3, above=0pt] {$\varepsilon$} (v1);
\draw[-,black, very thick, dotted] (vbp-2) edge (vb2);
\draw[->,black, very thick, shorten >=0.45cm] (vbp-1) edge  node[pos=0.7, below=0pt] {$\varepsilon$} (vp-2);

\draw[->,black, very thick, bend right=10] (vp-1) edge  node[pos=0.5, left=-3pt] {$\varepsilon$} (vbp-1);
\draw[->,black, very thick, bend right=10] (vbp-1) edge  node[pos=0.5, right=-3pt] {$0$} (vp-1);
\draw[->,black, very thick, bend right=30] (vb0) edge  node[pos=0.2, below=0pt] {$p-1$} (vbp-1);
\draw[->,black, very thick, bend left=30] (v0) edge  node[pos=0.2, above=0pt] {$\varepsilon$} (vp-1);
\draw[->,black, very thick, bend left=10] (v0) edge  node[pos=0.5, right=-3pt] {$\varepsilon$} (vb0);
\draw[->,black, very thick, bend left=10] (vb0) edge  node[pos=0.5, left=-3pt] {$p+1$} (v0);

\draw[->,black, very thick] (vx2) edge  node[pos=0.5, above=-2pt] {$\varepsilon$} (vx0);
\draw[->,black, very thick, shorten <=0.3cm] (vx4) edge  node[pos=0.65, above=-2pt] {$\varepsilon$} (vx2);
\draw[-,black, very thick, dotted] (vxp-4) edge (vx4);
\draw[->,black, very thick, shorten >=0.3cm] (vxp-2) edge  node[pos=0.35, above=-2pt] {$\varepsilon$} (vxp-4);

\draw[->,black, very thick] (vx3) edge  node[pos=0.5, below=-2pt] {$\varepsilon$} (vx1);
\draw[->,black, very thick, shorten <=0.3cm] (vx5) edge  node[pos=0.65, below=-2pt] {$\varepsilon$} (vx3);
\draw[-,black, very thick, dotted] (vxp-3) edge (vx5);
\draw[->,black, very thick, shorten >=0.3cm] (vxp-1) edge  node[pos=0.35, below=-2pt] {$\varepsilon$} (vxp-3);

\draw[->,black, very thick, loop right] (vxp-1) edge  node[pos=0.5, above=1pt] {$0$} (vxp-1);
\draw[->,black, very thick, loop left] (vx0) edge  node[pos=0.5, above=1pt] {$2$} (vx0);
\draw[->,black, very thick, bend left=25] (vx0) edge  node[pos=0.2, above=0pt] {$\varepsilon$} (vxp-2);
\draw[->,black, very thick, bend right=25] (vx1) edge  node[pos=0.2, below=0pt] {$0(p-1)0$} (vxp-1);
\draw[->,black, very thick] (vxp-1) edge  node[pos=0.5, right=-3pt] {$0$} (vxp-2);
\draw[->,black, very thick] (vx1) edge  node[pos=0.5, right=-3pt] {$20$} (vx0);

\draw[->,black, very thick] (vx0) edge  node[pos=0.5, above=-2pt,] {$\varepsilon$} (vxp-1);
\draw[->,black, very thick] (vx0.-30) -- (vxp-1.170)  node[pos=0.4, below=-3pt, rotate=-20] {$0(p-1)$};

\end{tikzpicture}

\caption{Prefix graph of the substitution $\zeta'_p$ for $p$ odd (left) and $p$ even (right). }
\label{zetapsgraph}
\end{figure}
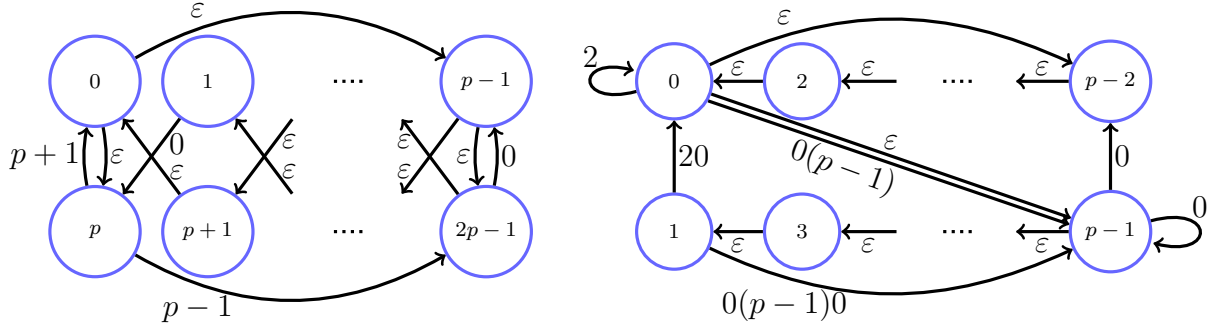

\begin{theorem}\label{iwip3}
Let $\alpha$ be a special Pisot unit such that $\alpha^2=\beta^p$ holds for an integer $p\geq 3$, i.e. $p \in \{3,4,6,8,10\}$.
Then $\zeta'_p$ (as defined in \eqref{setasta} and \eqref{setastb}, respectively) is a primitive unimodular Pisot substitution
with dominant root $\beta$ and 
such that for each $k \in \A$ we have $\R_{\zeta'_p}(k) \cong \Phi(X_k)$ with $X_k$ defined as follows.
For $p \in \{4,6,8,10\}$ we have
\[X_k = \begin{cases}
\bfPhi{\alpha\beta^k}(\Fa), & \text{if } k \in \{0, 2, \ldots, p-2\}; \\
\bfPsi{\beta^k} -\bfPhi{\alpha\beta^k}(\Fa), & \text{if } k \in \{1, 3, \ldots, p-3\}; \\
\bfPsi{\beta^{p-1}} - \bfPhi{\beta^{p-1}}(\Fa), & \text{if } k=p-1.
\end{cases}\]
If $p=3$ we have
\begin{align*}
X_0= & \bfPhi{\alpha}(\Fa),  & X_3= & \bfPsi{\alpha}- \bfPhi{\alpha} \Fa, \\
X_1= & \bfPsi{\beta}-\bfPhi{\alpha\beta}(\Fa), & X_4= & \bfPhi{\alpha\beta} \Fa, \\
X_2= & \bfPhi{\beta^2}(\Fa), & X_5= & \bfPsi{\beta^{2}} -\bfPhi{\beta^2} \Fa.
\end{align*} 
\end{theorem}
\begin{proof}
The strategy of the proof is the same as in the previous theorems.
We start with the case that $p$ is even. 
Obviously $\zeta'_p$ is a primitive substitution and $\bfM_{\zeta'_p} = \bfM_{\zeta_p}^2$ where $\zeta_p$ is the substitution discussed in Theorem~\ref{Rauzy}. 
From this we see that $\beta^2$ is the dominant root of $\bfM_{\zeta'_p}$ and   $\bfnu  = \left( \alpha, \beta, \beta^2 , \ldots, \beta^{p-1} \right)$ is a corresponding  left eigenvector. 
We define our isomorphism
$\Phi:\KK_\alpha \longmapsto \KK_{\zeta'_p}$ by the condition $\pi(x)=-\Phi\circ\bfPsi{x}$ for all $x \in \QQ(\alpha)$ and observe that it commutes
the action of $h$ with the action of $\bfPhi{\beta}^2$. With this we have
\[g_U \circ \Phi(\bfx)=  \Phi\left(-\bfPsi{\spk{\bfnu,\bfell(U)}}+\bfPhi{\beta^2}(\bfx)\right),\]
where $U \in \{\varepsilon, 0, 2, 20, 0(p-1), 0(p-1)0\}$ is one of the occurring prefixes. 

Again, we verify the $p$ set equations induced by the GIFS $(\Gamma_{\zeta'_p}, g)$. For $k \in \{2, 3, \ldots, p-2\}$ the calculations are trivial.
For $k=0$ we see that
\begin{align*}
\Phi(X_0)= & \Phi\circ\bfPhi{\alpha}(\Fa) = \Phi\circ\bfPhi{\alpha}\circ \fL(\Fa) \cup 
\Phi\circ\bfPhi{\alpha}\circ \fR\circ \fL(\Fa) \cup 
\Phi\circ\bfPhi{\alpha}\circ \fR^2(\Fa) \\
= &
\Phi\circ\bfPhi{\alpha^2}(\Fa) \cup \Phi\big(\bfPsi{\alpha\beta} - \bfPhi{\alpha^2\beta}(\Fa)\big) \cup \Phi\big(\bfPsi{\alpha\beta-\alpha\beta^2} + \bfPhi{\alpha\beta^2}(\Fa)\big) \\
= & g_\varepsilon\circ\Phi\circ\bfPhi{\beta^{p-2}}\circ\fL(\Fa) \cup 
g_\varepsilon\circ\Phi\circ\bfPhi{\beta^{p-2}}\circ\fR(\Fa) \\
&\cup 
\Phi\big(\bfPsi{-\alpha-\beta^{p-1}+\beta^{p+1}} - \bfPhi{\alpha^2\beta}(\Fa)\big)
\cup \Phi\big(\bfPsi{-\beta^2} + \bfPhi{\alpha\beta^2}(\Fa)\big) \\
= &  g_{\varepsilon} \circ \Phi(X_{p-2}) \cup  g_\varepsilon \circ \Phi(X_{p-1}) \cup  g_{0(p-1)} \circ \Phi(X_{p-1}) \cup g_2 \circ \Phi(X_0). 
\end{align*}

For $k=1$ we have
\[\Phi(X_1)= \Phi\big(\bfPsi{\beta} - \bfPhi{\alpha\beta} (\Fa)\big) = \Phi\big(\bfPsi{\beta} - \bfPhi{\alpha\beta}  \left(\fR(\Fa) \cup  \fL(\Fa)\right)\big).\]
Now observe that $\alpha^{-1}+\beta^{-1}=1$ we immediately yields $\beta^p=\alpha+\beta^{-1}$. 
From the characteristic polynomial of $M_{\zeta'_p}$ we obtain
$\beta=\beta^{p+1}-2\beta^p+\beta^{p-1}$. Insertion of the observation from above gives $\beta=-2\alpha-2\beta^{p-1}+\beta^{p+1}$.
We use this in our calculations and get
\begin{align*}
\Phi(X_1)&= \Phi\big(\bfPsi{-\alpha-\beta^2} +\bfPhi{\alpha\beta^2}(\Fa)\big) \cup  \Phi\big(\bfPsi{-2\alpha-\beta^{p-1}+\beta^{p+1}} -\bfPhi{\beta^{p+1}}(\Fa)\big)\\
&= g_{20} \circ \Phi(X_0) \cup g_{0(p-1)0} \circ \Phi(X_{p-1}).
\end{align*}

Finally, for the vertex $k=p-1$ we obtain
\begin{align*}
\Phi(X_{p-1}) = & \Phi\big(\bfPsi{\beta^{p-1}} - \bfPhi{\beta^{p-1}}\left(\fL( \Fa) \cup \fR(\Fa)\right)\big) \\
= & \Phi\big(\bfPsi{\beta^{p-1}} - \bfPhi{\alpha\beta^{p-1}}( \Fa)\big) 
\cup  \Phi\big(\bfPsi{\beta^{p-1}-\beta^p} + \bfPhi{\beta^{p}}( \Fa)\big) \\
= & g_\varepsilon \circ \Phi(X_{p-3}) \cup g_0 \circ \Phi\circ\bfPhi{\alpha\beta^{p-2}}\left(\fL(\Fa) \cup \fR(\Fa)\right) \\
= & g_\varepsilon \circ \Phi(X_{p-3}) \cup g_0 \circ \Phi(X_{p-2}) \cup  g_{0} \circ \Phi(X_{p-1}).
\end{align*}

Now lets turn to the case $p=3$. One easily verifies that  $\beta$ is the dominant root of $\bfM_{\zeta_3}$
and $\bfnu=(\alpha, \beta, \beta^2, \alpha, \beta, \beta^2)$ is a corresponding left eigenvector.
We thus define our isomorphism
$\Phi:\KK_\alpha \longmapsto \KK_{\zeta'_3}$ by the condition $\pi(x)=-\Phi\circ\bfPsi{x}$ for all $x \in \QQ(\alpha)$ and obtain
\[g_U \circ \Phi(\bfx)=  \Phi\left(-\bfPsi{\spk{\bfnu,\bfell(U)}}+\bfPhi{\beta}(\bfx)\right),\]
where $U \in \{\varepsilon,0,2,4\}$ is one of the occurring prefixes.

We verify the system of set equations~\eqref{setequations}
induced by the GIFS $(\Gamma_{\zeta'_3}, g)$ for the set list $\{X_0,\ldots, X_5\}$.
For $X_4$ the calculation is trivial. For the remaining sets we easily obtain
\begin{align*}
\Phi{X_0} & =  \Phi\big(\bfPhi{\alpha}\left(\fL(\Fa) \cup \fR(\Fa)\right)\big) = 
\Phi\big(\bfPhi{\alpha^2}(\Fa) \big) \cup \Phi\big(\bfPsi{\alpha\beta} - \bfPhi{\alpha\beta}(\Fa) \big)\\
& =  g_\varepsilon \circ \Phi(X_{2}) \cup  g_\varepsilon \circ \Phi(X_{3}), \\
\Phi(X_1)&=  \Psi\big(\bfPsi{\beta} -\bfPhi{\alpha\beta}(\Fa)\big) =
\big(\bfPsi{-\alpha+\alpha\beta} -\bfPhi{\alpha\beta}(\Fa)\big) =g_{0} \circ \Phi(X_{3}), \\
\Phi(X_{2})  & = \Psi\big(\bfPhi{\beta^2}(\Fa)\big) = g_\varepsilon \circ \Psi\big(\bfPhi{\beta}\left(\fL(\Fa)\cup\fR(\Fa)\right)\big) \\
& = g_\varepsilon \circ \Phi(X_{4}) \cup  g_\varepsilon \circ \Phi(X_{5}), \\
\Phi(X_{3})& = \Phi\big(\bfPsi{\alpha} - \bfPhi{\alpha}\left(\fR(\Fa) \cup \fL(\Fa) \right)\big) \\
&=\Phi\big(\bfPsi{\alpha-\alpha\beta} + \bfPhi{\alpha\beta}(\Fa)\big) \cup \Phi\big(\bfPsi{\alpha} - \bfPhi{\alpha^2}(\Fa)\big) \\
&=\Phi\big(\bfPsi{-\beta} + \bfPhi{\alpha\beta}(\Fa)\big) \cup \Phi\big(\bfPsi{-\beta^2+\beta^3} - \bfPhi{\beta^3}(\Fa)\big) \\
&= g_{4} \circ \Phi(X_{0}) \cup g_{2} \circ \Phi(X_{5}),\\
\Phi(X_{5}) &=  \Phi\big(\bfPsi{\beta^{2}} - \bfPhi{\beta^{2}}\left(\fL(\Fa) \cup \fR(\Fa) \right)\big) \\ 
&=  
\Phi\big(\bfPsi{\beta^{2}} - \bfPhi{\alpha\beta^{2}}(\Fa)\big) \cup \Phi\big(\bfPsi{\beta^{2}-\beta^3} + \bfPhi{\beta^{3}}(\Fa)\big) \\
&= g_{\varepsilon} \circ \Phi(X_{1}) \cup   g_0 \circ \Phi(X_{2}).
\end{align*}

\end{proof}

\begin{remark}\label{remark3}
Consider the  automorphism of the free group generated by $\{0, \ldots, p-1\}$
\[\eta'_p : \begin{cases}
0 \mapsto 1^{-1}0^{-1}, 1 \mapsto 2^{-1},\ldots, (p-2) \mapsto (p-1)^{-1}, (p-1) \mapsto 0(p-1)^{-1}, & \text{if $p$ is odd}; \\
0 \mapsto 1^{-1}0^{-1}, 1 \mapsto 2^{-1},\ldots, (p-2) \mapsto (p-1)^{-1}, (p-1) \mapsto (p-1)^{-1}0^{-1}, & \text{if $p$ is even}.
\end{cases}
\]
The substitution $\zeta'_p$ is the double substitution of $\eta'_p$. If $p$ is even then we have
$\zeta'_p= \eta'_p \circ \eta'_p$.
\end{remark}

\end{section}

\begin{section}{Proofs of the main results}\label{sec:tiling}

In the present section we prove our main results stated and announced in the introduction for higher dimensional tent-tiles (the case $d=1$ was already settled in Theorem~\ref{MT2-2}). At first we show the statement on the positive $d$-dimensional Lebesgue measure.
\begin{proof}[Proof of Theorem~\ref{positivemeasure}]
Due to the results in Section~\ref{sec:iwip} we can find for each special Pisot unit $\alpha$ a substitution $\zeta$  such
that the tent-tile $\Fa$ coincides - up to a linear isomorphism - with the associated Rauzy fractal $\R_\zeta$ or 
is contained in the invariant set list of the respective GIFS. Therefore, the statement follows immediately from
Proposition~\ref{Rauzypositivemeasure}.
\end{proof}

The results concerning the dimension of the boundary and the induced lattice tilings
will be discussed separately for each special Pisot unit. In these proofs the
eigenvalues of the adjacency matrices of the self-replicated boundary graph ($\mu_{\rm sr}$) and (in some cases) the lattice boundary graph ($\mu_{\rm lat}$)  for specific substitutions play an important role.
We will only state these  numbers and the number of vertices but we will not present the entire graphs
since this would go beyond the scope of the article.
Instead, we refer to \cite[Section~5.2]{Siegel-Thuswaldner.2009} where we find a detailed description how to 
algorithmically construct these graphs for any given substitution. From this instruction our stated results can easily be reproduced.
We also want to remark that in all the cases where we determine both the self-replicated boundary graph  and  the lattice boundary graph, the 
two graphs have a very similar shape, especially, we always have $\mu_{\rm lat} = \mu_{\rm sr}$. In fact, there 
does not seem to exist an example of a substitution (that satisfies \eqref{qmc}) such that $\mu_{\rm lat} \not= \mu_{\rm sr}$.

We performed the computations in Mathematica~\cite{Mathematica12.2}, the respective notebook files are available for download \cite{HPSurer}. There, the interested reader also finds more details concerning the calculated boundary graphs.

\subsection{Two dimensional tent-tiles}

The planar tent-tiles are associated with the six special Pisot numbers of algebraic degree $d+1=3$. They are depicted in Figure~\ref{planarcases} and one easily verifies that for each of them the respective Galois conjugates form a pair of complex conjugate numbers, that is  $r=0$ and $s=1$. Especially, the tent-tiles are contained in the complex plane.
\begin{remark}
For the production of the figures that show tent-tiles (Figure~\ref{planarcases} and Figure~\ref{tilings}) we chose the conjugates such that $\Im(\bfPsi{\alpha})>0$. 
Observe that this implies that $\bfPsi{\beta}$ has a negative imaginary part.
\end{remark}

\begin{theorem}\label{tiling1}
Let $\alpha=\alpha_1$, $\beta=\beta(\alpha)$ and
\[\Xi:=\{u_1\bfPsi{\beta-\alpha} + u_2\bfPsi{\beta-\alpha^2} : u_1, u_2 \in \ZZ\}.\]
Then
\[\dim_H(\partial \Fa) = \frac{2\log(\mu_{\rm sr})}{\log(\alpha)}
\approx 1.10026,\]
where $\mu_{\rm sr}\approx 1.3626$ is the dominant root of $t^5-2t^3+t-1$, and the collection $\{\bfx+\Fa : \bfx \in \Xi\}$ provides a proper tiling of the  space $\KK_\alpha=\CC$.
\end{theorem}
\begin{proof}
We consider the substitution $\zeta_3$ defined as in \eqref{zetap}.  The substitution is irreducible and satisfies the strong coincidence condition. 
The dominant roots of $\bfM_{\zeta_3}$ is given by $\lambda=\alpha \approx 1.75488$, the other roots form a complex conjugate pair $\lambda^{(1)}$, $\overline{\lambda^{(1)}}$ and we have
$\abs{\lambda^{(1)}}= \lambda^{-\nicefrac{1}{2}}$ since $\lambda$ is an algebraic unit. We algorithmically compute the self-replicating boundary graph $\Gamma^{(\rm sr)}_{\zeta_1}$ (depicted in Figure~\ref{graph}) and calculate the
dominant eigenvalue of the adjacency matrix $\mu_{\rm sr} \approx 1.3626$,
which is the dominant root of $t^5-2t^3+t-1$. For each $a \in \{0, \ldots, 9\}$ the set  
$\R_{\zeta_{3}}(a)$ is an affine image of $\Fa$ by Theorem~\ref{Rauzy}.
Therefore, 
$\dim_B(\partial\R_{\zeta_{3}}(a)) = \dim_B(\partial\Fa)$ and the statement on the Hausdorff dimension follows immediately from Corollary~\ref{Hdim}.

As $\zeta_{3}$ is irreducible, $\R_{\zeta_{3}}$ induces a lattice multi-tiling with respect to the lattice $\Xi_{\rm lat}$ which we may assume to be proper by the Pisot conjecture. For a confirmation we observe that $\mu_{\rm sr} < \lambda$. By
Theorem~\ref{tilingproperty}  this implies that the  self-replicating multi-tiling  is a proper tiling and since $\zeta_3$ is irreducible we
have that the lattice multi-tiling is also proper. 
Now, Theorem~\ref{Rauzy} immediately implies that $\Fa$ provides a proper tiling with respect to the lattice $\Xi$. 
\end{proof}

The left hand side of Figure~\ref{tilings} shows the structure of the tiling induced by the tent-tile $\F_{\alpha_1}$.
\begin{figure}[h]
\begin{minipage}[t]{0.4\textwidth}
\vspace{0cm}
  \includegraphics[width=0.68\textwidth]{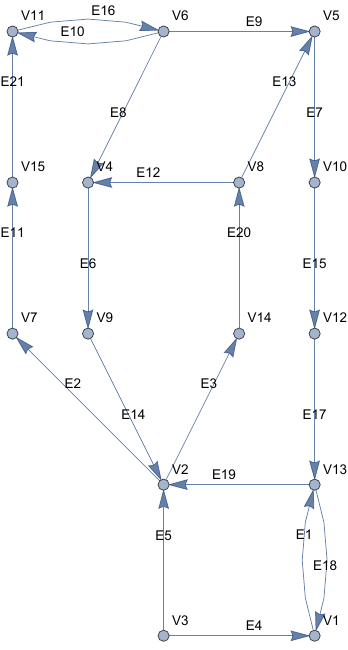}
\end{minipage}
\hfill
\begin{minipage}[t]{0.3\textwidth}
\vspace{0cm}
{\tiny
\begin{tabular}{l@{\qquad}l}
\multicolumn{2}{c}{Vertex List}\\
\hline
 V1& $[1 , \pi (0) , 2]$ \\
\hline
 V2& $[1 , \pi (0) , 3]$ \\
\hline
 V3& $[2 , \pi (0) , 3]$ \\
\hline
 V4& $[2 , \pi (\alpha ^2-2 \alpha +1) , 1]$ \\
\hline
 V5& $[3 , \pi (\alpha ^2-2 \alpha +1) , 1]$ \\
\hline
 V6& $[1 , \pi (\alpha -1) , 3]$ \\
\hline
 V7& $[3 , \pi (1) , 1]$ \\
\hline
 V8& $[3 , \pi (1) , 2]$ \\
\hline
 V9& $[1 , \pi (\alpha ^2-\alpha ) , 1]$ \\
\hline
 V10& $[2 , \pi (\alpha ^2-\alpha ) , 1]$ \\
\hline
 V11& $[2 , \pi (\alpha ^2-\alpha ) , 3]$ \\
\hline
 V12& $[1 , \pi (\alpha ) , 1]$ \\
\hline
 V13& $[1 , \pi (\alpha ) , 3]$ \\
\hline
 V14& $[3 , \pi (\alpha ) , 3]$ \\
\hline
 V15& $[3 , \pi (\alpha ^2-\alpha +1) , 3]$ \\
\hline
\end{tabular}}
\end{minipage}
\hfill
\begin{minipage}[t]{0.2\textwidth}
\vspace{0cm}
{\tiny
\begin{tabular}{l@{\qquad}l}
\multicolumn{2}{c}{Edge List}\\
\hline
E1 & $\pi (0)$ \\
\hline
E2 & $\pi (0)$ \\
\hline
E3 & $\pi (0)$ \\
\hline
E4 & $\pi (0)$ \\
\hline
E5 & $\pi (0)$ \\
\hline
E6 & $\pi (0)$ \\
\hline
E7 & $\pi (0)$ \\
\hline
E8 & $\pi (\alpha -1)$ \\
\hline
E9 & $\pi (\alpha -1)$ \\
\hline
E10 & $\pi (\alpha -1)$ \\
\hline
E11 & $\pi (0)$ \\
\hline
E12 & $\pi (0)$ \\
\hline
E13 & $\pi (0)$ \\
\hline
E14 & $\pi (\alpha ^2)$ \\
\hline
E15 & $\pi (0)$ \\
\hline
E16 & $\pi (0)$ \\
\hline
E17 & $\pi (\alpha )$ \\
\hline
E18 & $\pi (\alpha )$ \\
\hline
E19 & $\pi (\alpha )$ \\
\hline
E20 & $\pi (0)$ \\
\hline
E21 & $\pi (0)$ \\
\hline
\end{tabular}}
\end{minipage}
\caption{The self-replicating boundary graph
$\Gamma_{\zeta_3}^{({\rm sr})}$ has $15$ vertices. Note that the 
lattice boundary graph $\Gamma_{\zeta_3}^{({\rm lat})}$ corresponds to the subgraph 
obtained by removing the vertex V3 (and the edges E4 and E5).}
\label{graph}
\end{figure}

\begin{theorem}\label{tiling3}
Let $\alpha=\alpha_{3}$ and
\[\Xi:=\left\{u_1\bfPsi{1-\alpha} + u_2\bfPsi{1-\alpha^2} : u_1, u_2 \in \ZZ\right\}.\]
Then 
\[\dim_H(\partial \Fa) = \frac{2\log(\mu_{\rm sr})}{\log(\alpha)}
\approx 1.02952,\]
where $\mu_{\rm sr} \approx 1.21746$ is the dominant root of $t^7-2t^2-1$, and the collection $\{\bfx+\Fa : \bfx \in \Xi\} \cup \{\bfx+(\bfPsi{1}-\Fa) : \bfx \in \Xi\}$ provides a proper tiling of the space $\KK_\alpha=\CC$.
\end{theorem}
\begin{proof}
We consider the substitution $\theta_{3}$ discussed in Theorem~\ref{iwip1}. It is reducible and satisfies the strong coincidence condition.
The dominant root of $\bfM_{\theta_3}$ is given by $\lambda = \alpha \approx 1.46557$.
The statement concerning the Hausdorff dimension is shown analogously to Theorem~\ref{tiling1}.
We algorithmically compute the self-replicating boundary graph $\Gamma^{(\rm sr)}_{\theta_3}$ (41 vertices) and see that the largest eigenvalue of the incidence matrix $\mu_{\rm sr}$ is  the dominant root of $t^7-2t^2-1$.
Since $\R_{\theta_{3}}(a)$ is an affine image of $\Fa$ for each $a \in  \{0, \ldots, 5\}$ by Theorem~\ref{iwip1}
we can apply Corollary~\ref{Hdim} in order to calculate the Hausdorff dimension of the boundary of $\Fa$.

For the second part of the theorem we first observe that $\theta_{3}$ satisfies Condition~\eqref{qmc}. 
Algorithmic calculation of the lattice boundary graph $\Gamma^{({\rm lat})}_{\theta_3}$ (38 vertices) shows that
the largest eigenvalue of the incidence matrix is $\mu_{\rm lat} = \mu_{\rm sr} < \lambda$. From these observations we conclude that the Rauzy fractal $\R_{\theta_{3}}$ induces a proper lattice tiling with respect to the lattice $\Xi_{\rm lat}$.

In order to finish the proof we decompose each translate of the lattice tiling into the respective invariant sets and rearrange then cleverly in order to obtain the collection from the statement
 (see Figure~\ref{tilings3}).

\begin{figure}[h]
  \includegraphics[width=0.9\textwidth]{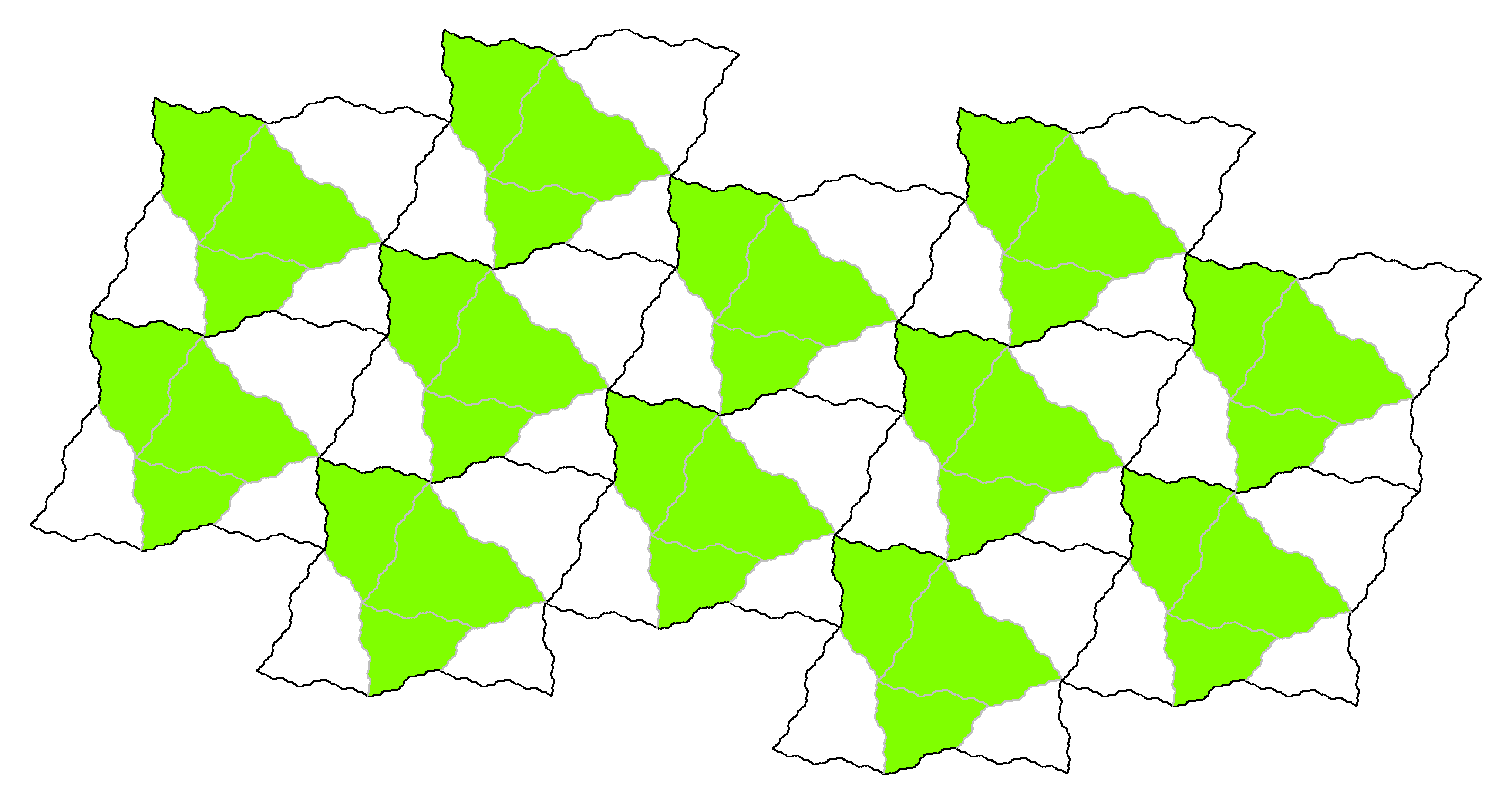}
\caption{The Rauzy fractal $\R_{\theta_3}$ (black boundaries) induces a lattice tiling. 
The boundaries of the invariant sets within $\R_{\theta_3}$ are drawn in grey.
The union of the invariant sets $\R_{\theta_3}(3)$, $\R_{\theta_3}(4)$, $\R_{\theta_3}(5)$ (coloured in green) correspond to the tent-tile $\F_{\alpha_3}$ by Corollary~\ref{iwip1cor}. By joining the other invariant sets across tiles in a suitable way we obtain reflections of $\F_{\alpha_3}$.}
\label{tilings3}
\end{figure}

In particular, since $\{\bfx+\R_{\theta_3} : \bfx \in \Xi_{\rm lat}\}$ is a proper tiling we conclude from
Lemma~\ref{coinc} that
\begin{multline*}
\big\{\bfx+\R_{\theta_3}(0): \bfx \in \Xi_{\rm lat}\big\} \cup \big\{\bfx+\R_{\theta_3}(1): \bfx \in \Xi_{\rm lat}\big\} \cup \big\{\bfx+\R_{\theta_3}(2): \bfx \in \Xi_{\rm lat}\big\}\\
\cup  \big\{\bfx+\big(\R_{\theta_3}(3) \cup \R_{\theta_3}(4) \cup\R_{\theta_3}(5)\big): \bfx \in \Xi_{\rm lat}\big\}
\end{multline*}
is also a proper tiling. We apply 
Theorem~\ref{iwip1} and Corollary~\ref{iwip1cor} and deduce that the collection
\begin{multline*}
\big\{\bfx+\bfPsi{1} - \fR(\Fa): \bfx \in \Xi\big\}
 \cup \big\{\bfx+\bfPsi{\alpha} - \fL\circ \fR(\Fa): \bfx \in \Xi\big\} \\
\cup \big\{\bfx+\bfPsi{\alpha^2} - \fL^2(\Fa): \bfx \in \Xi\big\} 
\cup  \big\{\bfx+\Fa: \bfx \in \Xi\big\}
\end{multline*}
is a proper tiling.
Since $\bfPsi{\alpha-1}$ and $\bfPsi{\alpha^2-1}$ are lattice points
we see that 
\begin{multline*}
\big\{\bfx+\big((\bfPsi{1} - \fR(\Fa)) \cup (\bfPsi{1} - \fL\circ \fR(\Fa)) \cup (\bfPsi{1} - \fL^2(\Fa))\big): \bfx \in \Xi\big\} \\
 \cup \big\{\bfx+\Fa: \bfx \in \Xi\big\}
\end{multline*}
is a proper tiling, too. Now, the statement of the theorem follows immediately by observing that
\begin{multline*}
(\bfPsi{1} - \fR(\Fa)) \cup (\bfPsi{1} - \fL\circ \fR(\Fa)) \cup (\bfPsi{1} - \fL^2(\Fa)) \\
= \bfPsi{1} - (\fR(\Fa) \cup  \fL\circ \fR(\Fa) \cup  \fL^2(\Fa)) 
=  \bfPsi{1} - \Fa.
\end{multline*}
\end{proof}

\begin{remark}
Observe that 
it also follows from \cite[Proposition~4.20]{Arnoux-Berthe-Hilion-Siegel.2006} (in terms of automorphisms of the free group)
that $\R_{\theta_3}$ induces a  lattice multi-tiling. The same  holds true for  Theorem~\ref{tiling-1}, Theorem~\ref{tiling-3} and
Theorem~\ref{tiling4} in an analogous way.
\end{remark}

\begin{theorem}\label{tiling-1}
Let $\alpha=\alpha_{-1}$, $\beta=\beta(\alpha)$ and
\[\Xi:=\left\{u_1\bfPsi{\alpha-\beta} + u_2\bfPsi{\alpha-\beta^2} : u_1, u_2 \in \ZZ\right\}.\]
Then 
\[\dim_H(\partial \Fa) = \frac{2\log(\mu_{\rm sr})}{\log(\beta)}
\approx 1.70018,\]
where $\mu_{\rm sr} \approx 1.61299$ is the dominant root of $t^{10}-t^7-4t^5-4t^4-t^3-4t^2-4t+1$, and the collection $\{\bfx+\Fa : \bfx \in \Xi\} \cup \{\bfx+(\bfPsi{\alpha}-\Fa) : \bfx \in \Xi\}$ provides a proper tiling of the  space $\KK_\alpha=\CC$.
\end{theorem}
\begin{proof}
The proof runs analogously to that of Theorem~\ref{tiling3}.
We consider the substitution $\zeta'_{3}$ discussed in Theorem~\ref{iwip3}. It is reducible, satisfies the weak coincidence condition and Condition~\eqref{qmc}.
The dominant root of $\bfM_{\zeta'_3}$ is given by $\lambda = \beta \approx 1.75488$.
We algorithmically compute the self-replicating boundary graph $\Gamma^{(\rm sr)}_{\zeta'_3}$, which has 47 vertices. It turns out that the largest eigenvalue of the incidence matrix $\mu_{\rm sr}$ is as stated.
As before, Corollary~\ref{Hdim} immediately yield $\dim_H(\partial \Fa)$.

We algorithmically calculate the lattice boundary graph $\Gamma^{(\rm lat)}_{\zeta'_3}$ (46 vertices) and determine 
the largest eigenvalue of the incidence matrix $\mu_{\rm lat} = \mu_{\rm sr} < \lambda$. 
Therefore, the collection $\{\bfx+\R_{\zeta'_{3}}: \bfx \in \Xi_{\rm lat}\}$ is a proper tiling.
From Lemma~\ref{coinc} we deduce that the collection
\begin{multline*}
\{\bfx+(\R_{\zeta_3}(0) \cup \R_{\zeta_3}(1) \cup\R_{\zeta_3}(2)): \bfx \in \Xi_{\rm lat}\}\\
\cup  \{\bfx+(\R_{\zeta_3}(3) \cup \R_{\zeta_3}(4) \cup\R_{\zeta_3}(5)): \bfx \in \Xi_{\rm lat}\}
\end{multline*}
is also a proper tiling. 
We apply Theorem~\ref{iwip3} and observe that $\bfPsi{\alpha-\beta}, \bfPsi{\beta-\beta^2} \in \Xi$. This shows that
\begin{multline*}
\{\bfx+(\bfPhi{\alpha}(\Fa) \cup (\bfPsi{\beta} - \bfPhi{\alpha\beta}(\Fa)) \cup (\bfPsi{\beta-\beta^2} - \bfPhi{\beta^2}(\Fa)): \bfx \in \Xi\} \cup \\
  \{\bfx+((\bfPsi{\alpha} -(\bfPhi{\alpha}(\Fa)) \cup (\bfPsi{\alpha-\beta} + \bfPhi{\alpha\beta}(\Fa)) \cup (\bfPsi{\alpha-\beta+\beta^2}- \bfPhi{\beta^2}(\Fa))): \bfx \in \Xi\}
\end{multline*}
is a proper tiling, too. 
From the observations
\begin{multline*}
\bfPhi{\alpha}(\Fa) \cup (\bfPsi{\beta} - \bfPhi{\alpha\beta}(\Fa)) \cup (\bfPsi{\beta-\beta^2} - \bfPhi{\beta^2}(\Fa)) \\
=\fL(\Fa) \cup \fR\circ \fL(\Fa) \cup f^2_R(\Fa) = \Fa,
\end{multline*}
\begin{multline*}
(\bfPsi{\alpha} -\bfPhi{\alpha}(\Fa)) \cup (\bfPsi{\alpha-\beta} + \bfPhi{\alpha\beta}(\Fa)) \cup (\bfPsi{\alpha-\beta+\beta^2}- \bfPhi{\beta^2}(\Fa)) \\
= (\bfPsi{\alpha} - \fL(\Fa)) \cup ((\bfPsi{\alpha} -\fR\circ \fL(\Fa)) \cup (\bfPsi{\alpha} - f^2_R(\Fa)) = \bfPsi{\alpha} -\Fa
\end{multline*}
the theorem follows immediately.
\end{proof}

\begin{theorem}\label{tiling-3}
Let $\alpha=\alpha_{-3}$, $\beta=\beta(\alpha)$ and 
\[\Xi:=\left\{u_1\bfPsi{1-\beta} + u_2\bfPsi{1-\beta^2} : u_1, u_2 \in \ZZ\right\}.\]
Then 
\[\dim_H(\partial \Fa) = \frac{2\log(\mu_{\rm sr})}{\log(\beta)}
\approx 1.25074,\]
where $\mu_{\rm sr} \approx 1.27004$ is the dominant root of $t^{10}-t^7-t^3-2t-1$,
and the collection $\{\bfx+\Fa : \bfx \in \Xi\} \cup \{\bfx+(\bfPsi{1}-\Fa) : \bfx \in \Xi\}$ is a proper tiling of the  space $\KK_\alpha=\CC$.
\end{theorem}
\begin{proof}
We use the well-known strategy from  Theorem~\ref{tiling3} and Theorem~\ref{tiling-1}.
Consider the substitution $\theta'_{3}$ discussed in Theorem~\ref{iwip2}. It is reducible and satisfies the weak coincidence condition as well as Condition~\eqref{qmc}.
The dominant root of $\bfM_{\theta'_3}$ is given by $\lambda=\beta  \approx 1.46557$.
We algorithmically compute the self-replicating boundary graph $\Gamma^{(\rm sr)}_{\theta'_3}$ (27 vertices) and realise that
the dominant eigenvalue of its adjacency matrix $\mu_{\rm sr}$ is the dominant root of $t^{10}-t^7-t^3-2t-1$.
Corollary~\ref{Hdim} yields the Hausdorff dimension $\dim_H(\partial \Fa)$.

The collection $\{\bfx+\R_{\theta'_{3}}: \bfx \in \Xi_{\rm lat}\}$
is a proper tiling since algorithmic calculation shows that $\mu_{\rm lat} = \mu_{\rm sr} < \lambda$ ($\Gamma^{(\rm lat)}_{\theta'_3}$ has 25 vertices). 
By observing Lemma~\ref{coinc}, Theorem~\ref{iwip2} and the fact that $\bfPsi{1-\beta}, \bfPsi{\beta-\beta^2} \in \Xi$
we see that 
\begin{multline*}
\{\bfx+(\bfPhi{\alpha}(\Fa) \cup (\bfPsi{\beta} - \bfPhi{\alpha\beta}(\Fa)) \cup (\bfPsi{\beta-\beta^2} - \bfPhi{\beta^2}(\Fa))): \bfx \in \Xi\} \cup \\
  \{\bfx+((\bfPsi{1} -\bfPhi{\alpha}(\Fa)) \cup (\bfPsi{1-\beta} + \bfPhi{\alpha\beta}(\Fa) \cup (\bfPsi{1-\beta+\beta^2}- \bfPhi{\beta^2}(\Fa))): \bfx \in \Xi\}
\end{multline*}
is also a proper tiling. From this one easily obtains the statement from the theorem.
\end{proof}
The right hand side of Figure~\ref{tilings} shows the structure of the tiling induced by the tent-tile $\F_{\alpha_{-3}}$.

Two tent-tiles do not seem to induce lattice tilings. Here we only calculate the Hausdorff dimension of the boundary which is done in the same way as in the previous theorems.

\begin{theorem}\label{TT5}
Let $\alpha=\alpha_{5}$. Then 
\[\dim_H(\partial \Fa) = \frac{2\log(\mu_{\rm sr})}{\log(\alpha)}
\approx 1.37858,\]
where $\mu_{\rm sr} \approx 1.21389$ is the dominant root of $t^{13}-t^7-t^6-2t^4+1$.
\end{theorem}
\begin{proof}
We consider the substitution $\theta_{5}$. 
It is reducible, satisfies the strong coincidence condition but does not satisfy Condition~\eqref{qmc} (therefore, $\R_{\theta'_{5}}$ does not induce a lattice multi-tiling). From Theorem~\ref{iwip1} we see that $\lambda= \alpha \approx 1.32472$ is the dominant root of $\bfM_{\theta_5}$. 

We algorithmically compute the self-replicating boundary graph $\Gamma^{(\rm sr)}_{\theta_5}$ (39 vertices).
It turns out that the dominant eigenvalue of its adjacency matrix is the dominant root of the polynomial $t^{13}-t^7-t^6-2t^4+1$.
For each $a \in \{0, \ldots, 9\}$ the invariant set  
$\R_{\theta_{5}}(a)$ is an affine image of $\Fa$, therefore, 
$\dim_B(\partial\R_{\theta_{5}}(a)) = \dim_B(\partial\Fa)$. Now, the statement of the theorem follows immediately from Corollary~\ref{Hdim}.
\end{proof}

\begin{theorem}\label{TT-5}
Let $\alpha=\alpha_{-5}$ and $\beta=\beta(\alpha)$.
Then 
\[\dim_H(\partial \Fa) = \frac{2\log(\mu_{\rm sr})}{\log(\beta)}
\approx 1.92089,\]
where $\mu_{\rm sr} \approx 1.31007$ is the dominant root of $t^{21}-t^{20}+t^{19}-t^{18}+t^{17}-2t^{16}+t^{15}-2t^{14}
+2t^{13}-3t^{12}+t^{11}+t^9-2t^8-3t^6-2t^5-2t^4-t^3-t^2+t+1$.
\end{theorem}
\begin{proof}
We consider the substitution $\theta'_{5}$. 
It is reducible and does not satisfy the strong coincidence condition. However, one easily verifies that the weak coincidence condition is satisfied. From Theorem~\ref{iwip2} we see that $\lambda=\beta  \approx 1.32472$ is the dominant root of $\bfM_{\theta'_5}$. 
Since Condition~\eqref{qmc} is not satisfied, $\R_{\theta'_{5}}$ does not induce a lattice multi-tiling. 

Without any problems one algorithmically calculates the self replicating boundary graph (151 vertices) and 
verify that the largest eigenvalue of its incidence matrix $\mu_{\rm sr}$ is exactly as claimed.
With this, $\dim_H(\partial \Fa)$ can be calculated from Corollary~\ref{Hdim}.
\end{proof}

\subsection{Three dimensional tent-tiles}

The  special Pisot numbers $\alpha_{-4}$ and $\alpha_4$ are of algebraic degree four, thus,  they induce
three-dimensional tent-tiles. For both cases the other Galois conjugates consist of a real number and a complex conjugate pair of numbers, that is  $r=s=1$. Especially, these conjugates have distinct modulus and therefore Corollary~\ref{Hdim} does only yield an upper bound for the Hausdorff dimension. However, the strategy we used for showing the results for planar tent-tiles also works for the three-dimensional case. 

\begin{theorem}\label{tiling4}
Let $\alpha=\alpha_{4}$ and
\[\Xi:=\left\{u_1\bfPsi{1-\alpha} + u_2\bfPsi{1-\alpha^2}+ u_3\bfPsi{1-\alpha^3} : u_1, u_2, u_3 \in \ZZ\right\}.\]
Then 
\[\dim_H(\partial \Fa) \leq \dim_B(\partial \Fa) = 3 +\frac{\log(\alpha) - \log(\mu_{\rm sr})}{\log\abs{\alpha^{(1)}}}
\approx 2.74421,\]
where $\mu_{\rm sr} \approx 1.31162$ is the dominant root of $t^{15}- 4 t^8- 2 t^6- 2 t^5- 2 t^3 -1$ and
$\alpha^{(1)} \approx -0.81917$ is the smallest Galois conjugate (according to modulus) of $\alpha$.
The collection $\{\bfx+\Fa : \bfx \in \Xi\} \cup \{\bfx+(\bfPsi{1}-\Fa) : \bfx \in \Xi\}$ provides a proper tiling of the  space $\KK_\alpha=\RR\times\CC$.
\end{theorem}
\begin{proof}
The proof runs quite analogous to that of Theorem~\ref{tiling3}.
We consider the substitution $\theta_{4}$ (see Theorem~\ref{iwip1}) which is reducible and satisfies the strong coincidence condition.
The dominant root of $\bfM_{\theta_4}$ is given by $\lambda = \alpha \approx 1.38028$.
For the Galois conjugates  $\lambda^{(1)} \in \RR$ and $\lambda^{(2)}$, $\overline{\lambda^{(2)}} \in \CC$ we have
$\abs{\lambda^{(1)}} \approx 0.81917 < \abs{\lambda^{(2)}} \approx 0.94044$.
Algorithmic computation of the self-replicating boundary graph $\Gamma^{(\rm sr)}_{\theta_4}$ (374 vertices) yields
$\mu_{\rm sr} \approx 1.31162$ (the dominant root of $t^{15}- 4 t^8- 2 t^6- 2 t^5- 2 t^3 -1$).
Now, the statement concerning the fractal dimension immediately follows from  Theorem~\ref{Bdim} and Corollary~\ref{Hdim}.

The collection $\{\bfx+\R_{\theta_{4}}: \bfx \in \Xi_{\rm lat}\}$
is a proper tiling since  $\theta_{4}$ satisfies Condition~\eqref{qmc} and algorithmic calculation shows that $\mu_{\rm lat} = \mu_{\rm sr} < \lambda_0$ ($\Gamma^{(\rm lat)}_{\theta_4}$ has 370 vertices).
From Lemma~\ref{coinc} we deduce that 
\begin{multline*}
\{\bfx+(\R_{\theta_{4}}(0) \cup \R_{\theta_{4}}(1) \cup \R_{\theta_{4}}(2) \cup \R_{\theta_{4}}(3)): \bfx \in \Xi_{\rm lat}\} \\
\cup 
\{\bfx+(\R_{\theta_{4}}(4) \cup \R_{\theta_{4}}(5) \cup \R_{\theta_{4}}(6) \cup \R_{\theta_{4}}(7)): \bfx \in \Xi_{\rm lat}\}
\end{multline*}
is a tiling.
We observe Theorem~\ref{iwip1}, Corollary~\ref{iwip1cor} and that
$\bfPsi{1-\alpha}, \bfPsi{1-\alpha^2}, \bfPsi{1-\alpha^3} \in \Xi$  in order to 
see that
\begin{multline*}
\{\bfx+((\bfPsi{1} - \fR(\Fa)) \cup (\bfPsi{1} - \fL\circ \fR(\Fa)) \cup \\
(\bfPsi{1} - \fL^2\circ \fR(\Fa)) \cup (\bfPsi{1} - \fL^3(\Fa))): \bfx \in \Xi\} 
\cup \{\bfx+\Fa: \bfx \in \Xi\}
\end{multline*}
is also a tiling and from this the theorem follows immediately.
\end{proof}

\begin{theorem}\label{tiling-4}
Let $\alpha=\alpha_{-4}$, $\beta=\beta(\alpha)$ and
\[\Xi:=\left\{u_1\bfPsi{1-\beta} + u_2\bfPsi{1-\beta^2}+ u_3\bfPsi{1-\beta^3} : u_1, u_2, u_3 \in \ZZ\right\}.\]
Then 
\[\dim_H(\partial \Fa) \leq \dim_B(\partial \Fa) = 3 +\frac{2\log(\beta) - \log(\mu_{\rm sr})}{2\log\abs{\beta^{(1)}}}
\approx 2.815,\]
where $\mu_{\rm sr} \approx 1.77033$ is the dominant root of $t^{28}- 2 t^{27}+ t^{26}- 6 t^{23}+ 2 t^{22}+ 6 t^{20}- 
 2 t^{19}+ 11 t^{18}+ t^{17} + 8 t^{16}- 14 t^{14}- 12 t^{13} - 5 t^{12}+ 
 9 t^{11}+ 14 t^{10}+ 8 t^9- t^8- 13 t^7+ 4 t^6- t^5- t^3+ t^2- 2 t+1$ and
$\beta^{(1)} \approx -0.81917$ is the smallest Galois conjugate (according to modulus) of $\beta$.
The  collection $\{\bfx+\Fa : \bfx \in \Xi\}$ provides a proper tiling of the space $\KK_\alpha=\RR\times\CC$.
\end{theorem}
\begin{proof}
We consider the substitution $\theta'_4$ discussed in Theorem~\ref{iwip2}. It is irreducible 
and satisfies the strong coincidence condition.
The dominant root of $\bfM_{\theta'_4}$ is given by $\lambda = \beta^2  \approx 1.90517$.
Algorithmic computation of the self-replicating boundary graph $\Gamma^{(\rm sr)}_{\theta_4}$ (119 vertices) yields 
$\mu_{\rm sr}$ as stated.
As in the previous theorem the fractal dimension can be calculated by applying Theorem~\ref{Bdim} and Corollary~\ref{Hdim}.

As $\theta'_4$ is irreducible $\R_{\theta'_4}$ induces a lattice multi-tiling which we may assume to be proper by the Pisot conjecture.
This is confirmed by Theorem~\ref{tilingproperty} since $\mu_{\rm sr} < \lambda_0$.
By a similar argumentation as in the previous theorems ($\bfPsi{\beta-\beta^2}$ is a lattice point) we conclude that
\begin{multline*}
\big\{\bfx + \big(\bfPhi{\alpha}(\Fa) \cup (\bfPsi{\beta} - \bfPhi{\alpha\beta}(\Fa))
\cup (\bfPsi{\beta-\beta^2} + \big(\bfPhi{\alpha\beta^2}(\Fa)) \\
\cup (\bfPsi{\beta-\beta^2+\beta^3} - \big(\bfPhi{\beta^3}(\Fa))\big)
: \bfx \in \Xi\big\}
\end{multline*}
is a proper tiling. By the definition of $\fR$ and $\fL$ we immediately obtain the stated result.
\end{proof}

\end{section}

\section*{Acknowledgement}
The authors would like to thank the referee who showed great interest in the research. The quite productive suggestions made it possible to significantly increase the quality of the article.


\begin{thebibliography}{10}

\bibitem{Akiyama-Barge-Berthe-Lee-Siegel:15}
{\sc S.~Akiyama, M.~Barge, V.~Berth{\'e}, J.-Y. Lee, and A.~Siegel}, {\em On
  the {P}isot substitution conjecture}, in Mathematics of Aperiodic Order,
  vol.~309 of Progress in Mathematics, Birkh\"{a}user, 2015, pp.~33--72.

\bibitem{Arnoux-Berthe-Hilion-Siegel.2006}
{\sc P.~{Arnoux}, V.~{Berth\'e}, A.~{Hilion}, and A.~{Siegel}}, {\em {Fractal
  representation of the attractive lamination of an automorphism of the free
  group.}}, {Ann. Inst. Fourier}, 56 (2006), pp.~2161--2212.

\bibitem{Arnoux-Ito.2001}
{\sc P.~Arnoux and S.~Ito}, {\em Pisot substitutions and {R}auzy fractals},
  Bull. Belg. Math. Soc. Simon Stevin, 8 (2001), pp.~181--207.
\newblock Journ\'ees Montoises d'Informatique Th\'eorique (Marne-la-Vall\'ee,
  2000).

\bibitem{Bandt.1989b}
{\sc C.~{Bandt}}, {\em {Self-similar sets. III: Constructions with sofic
  systems.}}, {Monatsh. Math.}, 108 (1989), pp.~89--102.

\bibitem{Barge.2018}
{\sc M.~{Barge}}, {\em {The Pisot conjecture for \(\beta\)-substitutions}},
  {Ergodic Theory Dyn. Syst.}, 38 (2018), pp.~444--472.

\bibitem{Barge-Diamond.2002}
{\sc M.~{Barge} and B.~{Diamond}}, {\em {Coincidence for substitutions of Pisot
  type}}, {Bull. Soc. Math. Fr.}, 130 (2002), pp.~619--626.

\bibitem{Barnsley-Vince.2014}
{\sc M.~{Barnsley} and A.~{Vince}}, {\em {Fractal tilings from iterated
  function systems.}}, {Discrete Comput. Geom.}, 51 (2014), pp.~729--752.

\bibitem{Berthe-Siegel.2005}
{\sc V.~{Berth\'e} and A.~{Siegel}}, {\em {Tilings associated with
  beta-numeration and substitutions}}, {Integers}, 5 (2005), pp.~paper a02, 46.

\bibitem{Berthe-Siegel-Thuswaldner.2010}
{\sc V.~Berth{\'e}, A.~Siegel, and J.~Thuswaldner}, {\em Substitutions, {R}auzy
  fractals and tilings}, in Combinatorics, automata and number theory, vol.~135
  of Encyclopedia Math. Appl., Cambridge Univ. Press, Cambridge, 2010,
  pp.~248--323.

\bibitem{Bestvina-Feighn-Handel.2000}
{\sc M.~{Bestvina}, M.~{Feighn}, and M.~{Handel}}, {\em {The Tits alternative
  for \(\text{Out}(F_n)\). I: Dynamics of exponentially-growing
  automorphisms}}, {Ann. Math. (2)}, 151 (2000), pp.~517--623.

\bibitem{Canterini-Siegel.2001b}
{\sc V.~Canterini and A.~Siegel}, {\em Geometric representation of
  substitutions of {P}isot type}, Trans. Amer. Math. Soc., 353 (2001),
  pp.~5121--5144 (electronic).

\bibitem{Ei-Ito-Rao.2006}
{\sc H.~Ei, S.~Ito, and H.~Rao}, {\em Atomic surfaces, tilings and
  coincidences. {II}. {R}educible case}, Ann. Inst. Fourier (Grenoble), 56
  (2006), pp.~2285--2313.

\bibitem{Fogg.2002}
{\sc N.~P. Fogg}, {\em Substitutions in dynamics, arithmetics and
  combinatorics}, vol.~1794 of Lecture Notes in Mathematics, Springer-Verlag,
  Berlin, 2002.
\newblock Edited by V.\ Berth\'e, S.\ Ferenczi, C.\ Mauduit and A.\ Siegel.

\bibitem{Hollander-Solomyak.2003}
{\sc M.~Hollander and B.~Solomyak}, {\em Two-symbol {P}isot substitutions have
  pure discrete spectrum}, Ergodic Theory Dynam. Systems, 23 (2003),
  pp.~533--540.

\bibitem{Hutchinson.1981}
{\sc J.~E. Hutchinson}, {\em Fractals and self-similarity}, Indiana Univ. Math.
  J., 30 (1981), pp.~713--747.

\bibitem{Ito-Rao.2006}
{\sc S.~Ito and H.~Rao}, {\em Atomic surfaces, tilings and coincidence. {I}.
  {I}rreducible case}, Israel J. Math., 153 (2006), pp.~129--155.

\bibitem{Lagarias.c2025}
{\sc J.~C. Lagarias}, {\em Personal communication}.
\newblock September 2025.

\bibitem{Lagarias-Porta-Stolarsky:93}
{\sc J.~C. Lagarias, H.~A. Porta, and K.~B. Stolarsky}, {\em Asymmetric tent
  map expansions. {I}. {E}ventually periodic points}, J. London Math. Soc. (2),
  47 (1993), pp.~542--556.

\bibitem{Lagarias-Porta-Stolarsky:94}
\leavevmode\vrule height 2pt depth -1.6pt width 23pt, {\em Asymmetric tent map
  expansions. {II}. {P}urely periodic points}, Illinois J. Math., 38 (1994),
  pp.~574--588.

\bibitem{Loridant-Messaoudi-Surer-Thuswaldner.2013}
{\sc B.~Loridant, A.~Messaoudi, P.~Surer, and J.~M. Thuswaldner}, {\em Tilings
  induced by a class of cubic {R}auzy fractals}, Theoret. Comput. Sci., 477
  (2013), pp.~6--31.

\bibitem{Mauldin-Williams:88}
{\sc R.~D. Mauldin and S.~C. Williams}, {\em Hausdorff dimension in graph
  directed constructions}, Trans. Amer. Math. Soc., 309 (1988), pp.~811--829.

\bibitem{Rao-Wen-Yang.2014}
{\sc H.~Rao, Z.-Y. Wen, and Y.-M. Yang}, {\em Dual systems of algebraic
  iterated function systems}, Adv. Math., 253 (2014), pp.~63--85.

\bibitem{Rauzy.1982}
{\sc G.~Rauzy}, {\em Nombres alg\'ebriques et substitutions}, Bull. Soc. Math.
  France, 110 (1982), pp.~147--178.

\bibitem{Scheicher-Sirvent-Surer:16}
{\sc K.~{Scheicher}, V.~F. {Sirvent}, and P.~{Surer}}, {\em {Dynamical
  properties of the tent map.}}, {J. Lond. Math. Soc., II. Ser.}, 93 (2016),
  pp.~319--340.

\bibitem{Siegel-Thuswaldner.2009}
{\sc A.~Siegel and J.~M. Thuswaldner}, {\em Topological properties of {R}auzy
  fractals}, M\'em. Soc. Math. Fr. (N.S.),  (2009), p.~140.

\bibitem{Sirvent-Wang.2002}
{\sc V.~F. Sirvent and Y.~Wang}, {\em Self-affine tiling via substitution
  dynamical systems and {R}auzy fractals}, Pacific J. Math., 206 (2002),
  pp.~465--485.

\bibitem{Smyth.1999}
{\sc C.~J. Smyth}, {\em There are only eleven special {P}isot numbers}, Bull.
  London Math. Soc., 31 (1999), pp.~1--5.

\bibitem{HPSurer}
{\sc P.~Surer}, {\em Online}.
\newblock \url{https://homepage.boku.ac.at/palovsky/TentTiles/}, 2024.

\bibitem{Mathematica12.2}
{\sc {Wolfram Research, Inc.}}, {\em Mathematica, {V}ersion 12.2}.
\newblock Champaign, IL, 2020.

\end{thebibliography}
\end{document}